\documentclass[12pt,reqno,a4paper]{amsart}
\usepackage{extsizes}
\usepackage{blindtext}
\usepackage{fullpage}
\usepackage{mathtools}
\usepackage{longtable}
\usepackage{amsmath,amssymb,amsthm}
\usepackage{amscd}
\usepackage{bm}
\usepackage{hyperref}

\usepackage{dsfont}
\usepackage{enumerate}
\usepackage{epsfig}
\usepackage{float, graphicx}
\usepackage{latexsym, amsxtra}
\usepackage{mathrsfs}
\usepackage{multicol}
\usepackage[normalem]{ulem}
\usepackage{psfrag}
\usepackage[parfill]{parskip}
\usepackage{stmaryrd}
\usepackage{tikz}
\usepackage[T1]{fontenc}
\usepackage{url}
\usepackage{verbatim}
\usepackage{indentfirst}
\usepackage{tikz-cd}

\usepackage{svg}
\usepackage[
ordering=Kac,
edge/.style=blue,
indefinite-edge={draw=green,fill=white,densely dashed},
indefinite-edge-ratio=5,
mark=o,
root-radius=.06cm]
{dynkin-diagrams}

\usepackage{mathtools}

\makeatletter
\def\thm@space@setup{%
	\thm@preskip=2ex \thm@postskip=2ex
}
\makeatother

\hypersetup{hidelinks}

\newtheorem{thm}{Theorem~}[section]
\newtheorem{lem}[thm]{Lemma~}

\newtheorem{prop}[thm]{Proposition~}

\newtheorem{cor}[thm]{Corollary~}

\newtheorem{conj}[thm]{Conjecture~}

\theoremstyle{remark}
\newtheorem{rmk}[thm]{Remark~}
\newtheorem{ex}[thm]{Example~}

\theoremstyle{definition}
\newtheorem{defn}[thm]{Definition~}

\newtheorem{cond}[thm]{Condition~}

\newcommand{\CC}{\mathbb{C}}
\newcommand{\ZZ}{\mathbb{Z}}

\newcommand{\LL}{\mathbb{L}}
\newcommand{\PP}{\mathbb{P}}

\newcommand{\QQ}{\mathbb{Q}}
\newcommand{\DD}{\mathbb{D}}
\newcommand{\HH}{\mathbb{H}}
\newcommand{\BB}{\mathbb{B}}

\newcommand{\Prd}{\mathscr{P}}

\newcommand{\X}{\mathscr{X}}
\newcommand{\Y}{\mathscr{Y}}

\newcommand{\sU}{\mathscr{U}}

\newcommand{\calM}{\mathcal{M}}

\newcommand{\calO}{\mathcal{O}}

\newcommand{\calT}{\mathcal{T}}
\newcommand{\calF}{\mathcal{F}}
\newcommand{\calN}{\mathcal{N}}

\newcommand\Aut{\mathrm{Aut}}

\newcommand\PU{\mathrm{PU}}

\newcommand\SL{\mathrm{SL}}

\newcommand\IV{\mathrm{IV}}

\newcommand\Spec{\mathrm{Spec}}

\newcommand{\Pic}{\mathrm{Pic}}

\newcommand{\Gr}{\mathrm{Gr}}
\newcommand{\GL}{\mathrm{GL}}

\newcommand{\Hom}{\mathrm{Hom}}

\newcommand{\IC}{\mathrm{IC}}

\newcommand{\bs}{\backslash}

\title{Calabi--Yau Varieties via Cyclic Covers, and Complex Hyperbolic Structures for their Moduli Spaces}
\vspace{1.2cm}
\author{Chenglong Yu, Zhiwei Zheng}
\date{}

\newcommand{\Addresses}{{
		\bigskip
		\footnotesize
		
		C.~Yu, \textsc{Center for Mathematics and Interdisciplinary Sciences, Fudan University and
Shanghai Institute for Mathematics and Interdisciplinary Sciences (SIMIS), Shanghai, China}\par\nopagebreak
		\textit{E-mail address}: \texttt{yuchenglong@simis.cn}
		
		\medskip
		
		Z.~Zheng, \textsc{Tsinghua University, Beijing, China}\par\nopagebreak
		\textit{E-mail address}: \texttt{zhengzhiwei@mail.tsinghua.edu.cn}
}}

\begin{document}
	\bibliographystyle{amsalpha}

	\begin{abstract}
		In this paper, we mainly study Calabi--Yau varieties that arise as triple covers of $(\PP^1)^n$ branched along simple normal crossing divisors. For some of those families of Calabi--Yau varieties, the period maps factor through arithmetic quotients of complex hyperbolic balls. We give a classification of such examples. One of the families was previously studied by Voisin, Borcea and Rohde. For these ball-type cases, we will show arithmeticity of the monodromy groups. These ball quotients are all commensurable to ball quotients in Deligne--Mostow theory. As a byproduct, we prove some commensurability relations among arithmetic groups in Deligne--Mostow theory.
	\end{abstract}
	
	\maketitle

\setcounter{tocdepth}{1}
\tableofcontents
	
\section{Introduction}
\label{section: introduction}
An important problem in algebraic geometry is to find locally symmetric varieties with modular interpretations. It is well known that the moduli spaces of polarized abelian varieties are arithmetic quotients of Siegel upper half spaces, and the moduli spaces of polarized $K3$ surfaces are arithmetic quotients of type IV domains. In \cite{deligne1986monodromy}, Deligne and Mostow considered moduli spaces of points on $\PP^1$, and found moduli spaces that are discrete quotients of complex hyperbolic balls.

There is a series of work realizing moduli spaces as arithmetic quotients of complex hyperbolic balls, for example, \cite{allcock2002complex, allcock2011moduli} on cubic surfaces and cubic threefolds, \cite{looijenga2007period} on cubic threefolds, \cite{kondo2000complex, kondo2000moduli} on curves of genus $3$ and $4$, \cite{yu2020moduli,yu2023moduli} on cubic fourfolds or nodal sextic curves with specified group actions. One common feature for most of these constructions is the so-called cyclic-cover method, namely, instead of the usual period map, one consider the Hodge structures of certain cyclic covers (with the original varieties being the branch locus). For example, Kond\=o \cite{kondo2000moduli} studied moduli spaces of curves of genus $4$ through periods of $K3$ surfaces that arise as triple covers of $(\PP^1)^2$. 

In this paper, we propose more constructions of Calabi--Yau covers such that the period maps factor through arithmetic ball quotients. Let $L$ be a line bundle on a smooth projective variety $Z$, and $V(L)$ the total space of $L$ with projection map $\pi_L\colon V(L)\to Z$. Let $d\ge 2$ be an integer, and suppose $s\in H^0(Z, L^d)$ is nonzero section of $L^d$ defining a normal crossing divisor $D=\sum\limits_{i=1}^m D_i$. Let $t$ be the tautological section of $\pi_L^*(L)$ over $V(L)$. Then the equation $t^d=\pi_L^*(s)$ defines a projective subvariety $Y$ in $V(L)$. The map $Y\to Z$ is a finite morphism of degree $d$ branched along $D$. There is an action of $\mu_d=\langle\exp({2\pi \sqrt{-1}\over d})\rangle$ on $Y$ given by multiplying $\exp({2\pi \sqrt{-1}\over d})$ on the variable $t$. Such covers $Y$ are projective orbifolds with finite quotient singularities, and their Hodge theory is well established; see \cite{steenbrink1977mixed}. We will review this in \S\ref{subsection: projective orbifold}.

Denote $L_i=\calO(D_i)$, then $L_1+\cdots+L_m=dL$ in $\Pic(Z)$. When the partition type $T=(L_1, \cdots, L_m)$ is fixed and divisors $D_1, \cdots, D_m$ are allowed to vary, we obtain a parameter space $\sU_T$ of normal crossing divisors of type $T$. If $K_Z=-(d-1)L$, the cover $Y$ is a Calabi--Yau orbifold. Consider the induced action of $\mu_d$ on $Y$. Then $\dim H^{n,0}(Y)=1$ and $H^{n,0}(Y)$ lies in the eigenspace $H^n_\chi(Y,\CC)\subset H^n(Y,\CC)$ where $\chi\colon \mu_d\hookrightarrow \CC^\times$ is the tautological embedding.

If $H^n_\chi(Y,\CC)=H^{n,0}(Y)\oplus H^{n-1, 1}_\chi(Y)$, then it is a Hodge structure of ball type, see \S\ref{subsection: HS of ball type}. In this case, the variation of Hodge structures over $\sU_T$ induces a period map with values in a complex hyperbolic ball (after taking quotient by the monodromy group). We then simply say that the period map for the cyclic covers $Y$ factors through a ball quotient. We also say that the partition $T$ is of ball type. The following are some known examples.
\begin{ex}[Sheng--Xu--Zuo]
\label{example: sheng}
Let $(Z, d, L)=(\PP^3, 3, \calO(2))$. Take the partition $T$ with $L_1=\cdots=L_6=\calO(1)$. In this case, we have $h^{2,1}=h_{\chi}^{2,1}=3$ and $h_{\chi}^{1,2}=0$, hence the period domain is a complex hyperbolic ball of dimension $3$. This case is studied in \cite{sheng2013maximal} and the global Torelli theorem is proved in \cite{sheng2019global}. When $D$ is not normal crossing, this construction gives different complete families of Calabi--Yau varieties with $h^{2,1}=h^{2,1}_{\chi}=2$ or $1, 0$. For $(Z, d, L)=(\PP^3, 3, \calO(2))$, all the other partitions of $\calO(6)$ do not give ball-type period domains. See \S\ref{section: case Pn} for more details.
\end{ex}

\begin{ex}[Voisin--Borcea--Rohde]
\label{example: voisin}
 Let $(Z, d, L)=((\PP^1)^3, 3, \calO(3))$. Take the partition $T$ with $L_1=\calO(3,3,0)$ and $L_2=\calO(0,0,3)$. Then the cyclic cover $Y$ is a degree-$3$ quotient of a product of a K3 surface and an elliptic curve. Then the period domain is a complex hyperbolic ball of dimension $9$, and the global Torelli holds. This was previously studied by Rohde \cite{rohde2009cyclic}. The method was originally developed by Voisin \cite{voisin1993miroirs} and Borcea \cite{borcea1997k3} when they constructed Calabi--Yau varieties using involution of K3 surfaces.
\end{ex}

Before stating our main result, we introduce the notion of half-twists.

\begin{defn} 
\label{definition: fermat half twist}
Let $d=3$ and $T=(L_1,\cdots,L_m)$ be a partition of $3L$ on $Z$. The Fermat-type half-twist of ${T}$ is a partition of $3L\boxtimes \calO_{\PP^1}(3)$ on $Z\times \PP^1$ defined by $\widetilde{T}=(L_1\boxtimes \calO_{\PP^1}, \cdots, L_m\boxtimes \calO_{\PP^1}, \calO_{Z}\boxtimes \calO_{\PP^1}(3))$.
\end{defn}

Fermat-type half-twist construction transforms ball-type examples from dimension $n-1$ to $n$, unchanging the dimension of balls; see Proposition \ref{proposition: half twist ball}. It corresponds to van Geemen's half-twist construction for Hodge structures; see \cite{van2001half}\cite{van2002half}. We will explain more details in \S\ref{subsection: half-twist}. 

The first main result of this paper is an explicit classification of ball-type partitions of $dL$ for $(Z, d, L)=((\PP^1)^n, 3, \calO(1)^{\boxtimes n})$. All ball-type cases for $n\geq 3$ can be obtained by taking Fermat-type half-twists or refinements of four cases in dimension $n=3$ and one case in dimension $n=4$.

\begin{thm}
\label{theorem: main}
Let $(Z, d, L)=((\PP^1)^n, 3, \calO(1)^{\boxtimes n})$. Ball-type partitions $T=(L_1,\cdots, L_m)$ for $3L$ can be described as follows up to permutations.

When $n=2$, all partitions give ball-type examples.

When $n=3$, each part $L_j$ of $T$ has at least one zero component. Equivalently $T$ is a refined partition of the following four:
\begin{enumerate}[(i)]
\item $L_1=\calO(3,3,0)$ and $L_2=\calO(0,0,3)$;
\item $L_1=\calO(3,2,0)$ and $L_2=\calO(0,1,3)$;
\item $L_1=\calO(2,2,0)$, $L_2=\calO(1,0,2)$ and $L_3=\calO(0,1,1)$;
\item $L_1=\calO(2,1,0)$, $L_2=\calO(1,0,2)$ and $L_3=\calO(0,2,1)$.
\end{enumerate}
The dimensions of the complex hyperbolic balls for the above four maximal cases are $9,9,7,6$, respectively.

When $n=4$, the partition $T$ is either a refinement of 
\begin{equation*}
L_1=\calO(1,3,0,0), L_2=\calO(1,0,3,0), L_3=\calO(1,0,0,3)
\end{equation*}
(the ball for this maximal case has dimension $9$) or a Fermat-type half-twist of a ball-type partition for $n=3$.

When $n\geq 5$, the partition $T$ is
a Fermat-type half-twist of a ball-type partition for $n-1$.
\end{thm}

The case $(i)$ in Theorem \ref{theorem: main} was previously studied by Voisin, Borcea, and Rohde; see Example \ref{example: voisin}. 

We now introduce our approach to find examples with ball-type period domains. When $Z$ is regular and rigid, and $Y$ is Calabi--Yau, we show that the variation of Hodge structures of $Y$ over $\sU_T$ satisfies certain infinitesimal Torelli. This allows us to calculate the dimension of $H^{n-1,1}_\chi$; see Theorem \ref{theorem: local torelli}. In particular, when $Z$ is a product of projective spaces, a generic point of $\sU_T$ is stable and we have a GIT quotient $\calM_T$ of $\sU_T$ by $\Aut(Z)$; see Proposition \ref{proposition: stability}. Then $\dim H^{n-1,1}_\chi=\dim \calM_T$.

If $H^{n-p,p}_\chi(Y)=0$ for $p\ge 2$ (which is equivalent to $\dim H^n_\chi(Y,\CC)=1+\dim H^{n-1,1}_\chi$), then the eigenspace $H_\chi^n(Y,\CC)=H_\chi^{n,0}\oplus H_\chi^{n-1,1}$ is in a natural way a unitary Hermitian form, and the period map factors through a ball quotient. Therefore, we aim to find cases when $\dim H^n_\chi=1+\dim H^{n-1,1}_\chi$. We hence need to calculate $\dim H^n_\chi$ and $\dim H^{n-1,1}_\chi=\dim \calM_T$. This is accomplished in \S\ref{section: main} for the cases $(Z,d,L)=((\PP^1)^n, 3, \calO(1)^{\boxtimes n}), (\PP^3, 3, \calO(2))$ and $(\PP^3, 5, \calO(1))$.

As pointed out by Prof. Chin-Lung Wang, the Weil--Petersson metric on the moduli of Calabi--Yau manifolds is complex hyperbolic if and only if the Yukawa coupling length is minimal; see \cite[Theorem 2.1]{wang2003curvature}. Especially, when $H^{n-2,2}_\chi(Y)=0$, the GIT quotient $\calM_T$ has a natural complex hyperbolic structure.

We emphasize the relation between our constructions and Deligne--Mostow theory. The relation between $K3$ surfaces with isotrivial fibrations and Deligne--Mostow theory has been noticed and investigated in \cite{dolgachev2007moduli}\cite{moonen2018Deligne--Mostow}\cite{zhong2022ball}. In case $((\PP^1)^2, 3, \calO(1)^{\boxtimes 2})$, a simple normal crossing curve of degree $(3,3)$ gives rise to a $K3$ surface with two isotrivial elliptic fibrations, which correspond to two Deligne--Mostow data. Based on this, we prove that some Deligne--Mostow lattices are actually commensurable; see Corollary \ref{corollary: comm}. We will generalize this to higher dimensional cases in \S\ref{section: 3 DM}. For example, we show that the monodromy groups in both cases $(i), (ii)$ of Theorem \ref{theorem: main} are commensurable to the arithmetic group in Deligne--Mostow theory with type $(({1\over 6})^{12})$, see Corollary \ref{corollary: comm}.

This paper is organized as follows. In \S\ref{section: infinitesimal torelli} we study the variation of Hodge structures of cyclic covers branching along normal crossing divisors, and prove the infinitesimal Torelli theorem under Condition \ref{condition}. In \S\ref{section: main} we proved Theorem \ref{theorem: main} with detailed calculation. In \S\ref{section: refinement, incidence and complete} we use a generalization of the Clemens-Schmid sequence obtained by Kerr-Laza \cite{kerr2021hodge} to study the Hodge degeneration associated with a refinement. In particular, we show that a refinement of a ball-type refinement is still of ball-type; see Theorem \ref{theorem: refinement}. In \S\ref{section: completeness} we classify ball-type partitions $T$ for $Z=(\PP^1)^n$ such that the family of crepant resolutions of $Y$ over $\sU_T$ is complete. In \S\ref{section: period map} we prove that the monodromy groups for the ball-type cases obtained in Theorem \ref{theorem: main} are arithmetic. In \S\ref{section: relation to DM} and \S\ref{section: 3 DM} we investigate the relation between our constructions and Deligne--Mostow theory.
	
\noindent \textbf{Acknowledgement}: The first author is supported by the national key research and development program of China (No. 2022YFA1007100) and NSFC 12201337. The second author is partially supported by NSFC 12301058. This work started when the second author was a postdoc at the Max Planck Institute of Mathematics in 2020. The second author thanks MPIM for its support. The authors thank Dingxin Zhang for discussion on calculation of Hodge numbers. The authors thank Chin-Lung Wang for pointing out the relation between curvature of Weil--Petersson metrics and Yukawa coupling. The authors thank Bong Lian, Mao Sheng, Chin-Lung Wang and Kang Zuo for their interests and helpful discussion.

\section{Infinitesimal Torelli}
\label{section: infinitesimal torelli}

In this section, we study the variation of Hodge structures induced by a family of cyclic covers of a fixed smooth projective variety branching along simple normal crossing divisors, and prove a version of infinitesimal Torelli theorem (Theorem \ref{theorem: local torelli}).

\subsection{Projective Orbifolds}
\label{subsection: projective orbifold}
We first review the fundamental facts on Hodge theory of projective orbifolds and collect some useful facts on cyclic covers. All varieties in this paper are defined over the complex field $\CC$.
	
A projective orbifold is a projective variety that is locally analytically isomorphic to an open neighborhood of $0$ in $\CC^n/G$ with $G$ a finite subgroup of $\GL(n, \CC)$. We review the Hodge theory for projective orbifolds, which was developed by Bailey \cite{bailey1957imbedding} and Steenbrink \cite{steenbrink1977mixed}, see also \cite[Theorem 2.43]{peters2008mixed} and \cite[Appendix A.3]{cox1999mirror}. For a projective orbifold $Y$, let $Y_0$ be the smooth part of $Y$. Let $\widehat{\Omega}_Y^p=(\Omega_Y^p)^{**}$ be the reflexive hull of $\Omega_Y^p$. Let $j\colon Y_0\hookrightarrow Y$ be the natural inclusion. In this case, we have $\widehat{\Omega}^p_Y=j_*(\Omega^p_{Y_0})$, see \cite[page 507]{arapura2014hodge}. The singular cohomology $H^*(Y, \QQ)$ has a pure Hodge structure and admits Hodge decomposition $H^*(Y, \CC)=\bigoplus H^{p,q}(Y)$ with $H^{p,q}(Y)\cong H^q(Y, \widehat{\Omega}^p_Y)$.

Given a smooth projective variety $Z$ of dimension $n$, an integer $d\ge 2$, and a line bundle $L$ on $Z$. Suppose that $D$ is an effective divisor defined by section $s$ of $L^d$.
\begin{defn}
Let $\calF=\calO_Z\oplus L^{-1}\oplus L^{-2}\oplus \cdots \oplus L^{-d+1}$ be an $\calO_Z$-algebra with the algebra structures given by the morphism $L^{-d}\to \calO_Z$, $x\mapsto sx$. Then $\Spec(\calF) \to Z$ is a finite morphism of degree $d$ ramified at $D$. We call the normalization of $\Spec(\calF)$ the degree-$d$ cyclic cover of $Z$ branching along $D$.
\end{defn}

Suppose that we have a decomposition $dL=L_1+\cdots+L_m$ in $\Pic(Z)$. Denote $T=(L_1, \cdots, L_m)$ and call it the partition type. We consider tuples $(Z, d, L, T)$ such that there exist sections $s_i\in H^0(Z, L_i)$ with $s_i$ defining a smooth divisor $D_i$, and $s=s_1 s_2\cdots s_m\in H^0(Z, L)$ defining a simple normal crossing divisor $D=\sum\limits_{i=1}^m D_i$. Then $Y=\Spec(\calF)$ is a projective orbifold and automatically normal. Therefore, we have a cyclic cover $\pi\colon Y\to Z$ of degree $d$ branching along $D$. This definition is equivalent with the one we discussed in the introduction; see also \cite[Lemma 1.1]{arapura2014hodge}. In this paper, we study the cyclic covers that give rise to Calabi--Yau orbifolds.
	
\begin{defn}
\label{definition: CY orbifold}
Let $Y$ be a projective orbifold of dimension $n$. Suppose $H^0(Y, \widehat{\Omega}^k_{Y})=0$ for $k=1,\cdots, n-1$ and $\widehat{\Omega}^n_{Y} \cong \calO_{Y}$, then we call $Y$ a Calabi--Yau orbifold.
\end{defn}

\begin{lem}
Let $(Z, d, L, D)$ be the data associated with a cyclic cover $Y\to Z$ of degree $d$, such that $Z$ does not admit nontrivial holomorphic forms and $D$ is simple normal crossing. Suppose that $(d-1)L=-K_Z$. Then $Y$ is a projective Calabi--Yau orbifold.
\end{lem}
\begin{proof}
By \cite[Corollary 1.6]{arapura2014hodge}, there is an isomorphism $H^i(Z, \QQ)\cong H^i(Y, \QQ)$ for $0\le i\le n-1$. Thus $H^{i,0}(Y)\cong H^{i,0}(Z)=0$. By the adjunction formula, the canonical bundle $K_Y\cong \pi^*(K_Z\otimes L^{d-1})$ is trivial.
\end{proof}

\subsection{Period domain}
Let $Z$ be a smooth projective variety, $L\in \Pic(Z)$ and $d\ge 2$ be an integer such that $-K_Z=(d-1)L$. Assume that we have a tuple $(Z,d,L, L_1, \cdots, L_m)$ such that there exists $s_i\in H^0(Z, L_i)$ defining a smooth divisor $D_i$, and $s=s_1\cdots s_k$ defining a simple normal crossing divisor $D=\sum\limits_{i=1}^m D_i$. 

Denote $T=(L_1, \cdots, L_m)$. Denote $\PP_T=\prod_{i=1}^k\PP(H^0(Z, L_i))$ the product of linear systems of $L_i$. Let $\sU_T$ be the Zariski open subset of $\PP_T$ consisting of simple normal crossing divisors of type $T$. The degree-$d$ cyclic covers of $Z$ branching along the divisors in $\sU_T$ form a family $\Y_T \to \sU_T$ of singular Calabi--Yau varieties. This family admits a simultaneous crepant resolution, see \S\ref{section: completeness}. The pure Hodge structures on the integral middle cohomology of the fibers of $\Y_T \to \sU_T$ form a VHS (variation of Hodge structures) $\HH_T\to \mathscr{U}_T$. This VHS admits an induced action of the cyclic group $\mu_d$.

Recall that $\chi\colon \mu_d\hookrightarrow \CC^{\times}$ is the tautological character of $\mu_d$. Let $\xi_d=\exp(\frac{2\pi\sqrt{-1}}{d})$. The middle cohomology group $H^n(Y, \QQ[\xi_d])$ admits an induced action of $\mu_d$. Denote by $H^n_\chi(Y,\QQ[\xi_d])\subset H^n(Y, \QQ[\xi_d])$ the $\chi$-eigenspace with respect to this action. We have $H^{n,0}(Y)\subset H^n_\chi(Y, \CC)$, see \cite[Lemma 1.2]{arapura2014hodge}. Denote by $h^{p,q}_\chi=\dim  H^{p,q}_\chi (Y, \CC)$. The Hodge structures of type $(1, h^{n-1,1}_{\chi},\cdots, h^{1,n-1}_{\chi}, h^{0,n}_{\chi})$ on $H^n_{\chi}(Y, \QQ[\xi_d])$ form a sub-VHS of $\HH_T$, which we denote by $H_{T,\chi}$. 

There is a polarization on $H^n_\chi(Y,\QQ[\xi_d])$ given by 
\begin{equation*}
H^n_\chi(Y,\QQ[\xi_d]) \times H^n_\chi(Y,\QQ[\xi_d])\to \QQ[\xi_d]
\end{equation*}
sending $(x,y)$ to $\int_Y x\overline{y}$. Let $\BB_T, \Gamma_T$ be the period domain and monodromy group for the VHS $\HH_{T,\chi}\to \sU_T$ together with the polarization.  Then we have an analytic map $\Prd\colon \mathscr{U}_T\longrightarrow \Gamma_T\bs\BB_T$. We call $\Prd$ the period map for the Calabi--Yau orbifolds $Y$.

One of our main goals is to find $(Z, d, L, T)$ with the period domain of ball type. This is true when $h^{n-2,2}_\chi(Y)=h^{n-3,3}_\chi(Y)=\cdots=h^{0,n}_\chi(Y)=0$. We will discuss this in Proposition \ref{proposition: ball type}.
 
\subsection{Infinitesimal Torelli}
Now we state and prove the infinitesimal Torelli theorem that we need for further calculation. We summarize the conditions we need for $(Z, d, L, T)$ in the following:
\begin{cond}
\label{condition}
\begin{enumerate}[(i)]
\item $Z$ is a smooth projective  variety with $H^{i,0}(Z)=0$ for $1\le i\le n-1$,
\item $Z$ is rigid, namely $H^0(Z, \calT_Z)=0$, 
\item $d\ge 3$,
\item there exists simple normal crossing divisor of type $T$, 
\item $-K_Z=(d-1)L$. 
\end{enumerate}
\end{cond}

Item (i) in Condition \ref{condition} implies $H^1(Z,\CC)=0$, which ensures that the identity component $\Aut^\circ (Z)$ of the regular automorphism group $\Aut(Z)$ naturally acts on the complete linear system of any line bundle on $Z$. The triviality of $H^1(Z, \calT_Z)$ means that $Z$ is rigid. Item (iv) in Condition \ref{condition} implies that $\sU_T$ is a nonempty Zariski open subspace of $\PP_T$.

\begin{thm}
\label{theorem: local torelli}
Suppose $(Z,d,L,T)$ satisfies Condition \ref{condition}. Then the tangent map of the period map $\Prd\colon \sU_T\to \BB_T/\Gamma_T$ at $D\in \mathscr{U}_T$ is $dp\colon T_{D}\mathscr{U}_T\to \Hom(H^{n,0}(Y), H^{n-1,1}_\chi(Y,\CC))$, and it fits into an exact sequence 
 \begin{equation*}
 H^0(Z, \calT_Z)\rightarrow T_{D}\mathscr{U}_T\xrightarrow[]{dp} \Hom(H^{n,0}, H^{n-1,1}_\chi) \to 0
 \end{equation*}
 Especially, if the stabilizer of $D\in \mathscr{U}_T$ under $\Aut^0(Z)$ is a discrete group, then $\dim(\mathscr{U}_T)-\dim \Aut(Z)=h^{n-1,1}_\chi(Y)$.
\end{thm}
	
\begin{proof}
We first show the proof when the partition of $L^d$ is trivial, which means $m=1$. In this case $D$ is smooth and $Y$ is smooth. Closed embedding $j\colon D\to Z$ induces a short exact sequence
\begin{equation*}
0\to \calT_Z(-\log D)\to \calT_Z\to j_*\calN_{D/Z}\to 0,
\end{equation*}
where $\calN_{D/Z}$ is the normal bundle of $D$ in $Z$ and is isomorphic to $j^*\calO_Z(D)\cong j^*L^d$. The cohomology long exact sequence gives
\begin{equation*}
H^0(Z, \calT_Z(-\log D))\to H^0(Z, \calT_Z) \to H^0(Z, j_*\calN_{D/Z})\to H^1(Z, \calT_Z(-\log D))\to H^1(Z,\calT_Z)=0.
\end{equation*}
The last equality is by rigidity of $Z$ (item (ii) in Condition \ref{condition}). Next we show $H^0(Z, j_*\calN_{D/Z})\cong T_D \sU_T$ and $H^1(Z, \calT_Z(-\log D))\cong \Hom(H^{n,0}, H^{n-1,1}_\chi)$ in a natural way.	

The space $H^0(Z, j_*\calN_{D/Z})\cong H^0(D, L^d|_D)$ represents the infinitesimal deformation of $D$ in $Z$ and the image of $H^0(Z, \calT_Z)$ represents the deformation of $D$ induced by the action of $\Aut^\circ(Z)$. From the short exact sequence
		\begin{equation*}
			0\to \calO_Z\to \calO_Z(D) \to j_*j^*\calO_Z(D)\to 0,
		\end{equation*}
  we have the following long exact sequence (notice that $L^d=\calO_Z(D)$):
  \begin{equation*}
\CC\to H^0(Z, L^d)\to H^0(D, L^d|_D)\to H^1(Z, \calO_Z)=0
  \end{equation*}
So $H^0(D, L^d|_D)\cong  H^0(Z, L^d)/\CC$ is naturally identified with the tangent space $T_D\mathscr{U}_T$ of $\mathscr{U}_T\subset \PP(H^0(Z, L^d))$ at the point $D$. 

Denote by $\pi\colon Y\to Z$ the cyclic covering map. By \cite{esnault1992lectures}, the $\mu_d$-invariant part of $\pi_*(\calT_Y)$ is isomorphic to $\calT_Z(-\log D)$. So $H^1(Y, \calT_Y)^{\mu_d}\cong H^1(Z, \calT_Z(-\log D))$. Since $Y$ is a smooth Calabi--Yau manifold, we have the Kodaira-Spencer map 
\begin{equation*}
H^1(Y, \calT_Y)\to \Hom(H^{n,0}(Y), H^{n-1,1}(Y)),
\end{equation*}
which is an isomorphism. Since $H^{n,0}(Y)=H^{n,0}_\chi(Y)$, we have an isomorphism $H^1(Y, \calT_Y)^{\mu_d}\cong \Hom(H^{n,0}_\chi(Y), H^{n-1,1}_\chi(Y))$.

The proof for the general case follows from the same idea. Assume that $D=\sum\limits_{i=1}^k D_i$ and all components $D_i$ are smooth and intersect transversely. Denote $j_i\colon D_i\hookrightarrow Z$. We have the following short exact sequence.
		\begin{equation*}
			0\to \calT_Z(-\log D)\to \calT_Z\to \bigoplus_i {j_i}_*\calN_{D_i/Z}\to 0.
		\end{equation*}
		The induced long exact sequence gives the following.
		\begin{equation*}
			H^0(Z, \calT_Z)\to \bigoplus_i H^0(D_i, \calN_{D_i/Z})\to H^1(Z, \calT_Z(-\log D))\to 0.
		\end{equation*}
  From the same argument, $\bigoplus\limits_i H^0(D_i, \calN_{D_i/Z})\cong T_{D}\mathscr{U}_T$.
		Hence, 
  \begin{equation*}
  H^0(Z, \calT_Z)\to T_{D}\mathscr{U}_T\to H^1(Z, \calT_Z(-\log D))\to 0
  \end{equation*}
  is a long exact sequence. From \cite[\S2]{kawamata1985minimal}, \cite{esnault1992lectures}, \cite[Theorem 1.1]{zhang2012introduction}, we still have $H^1(Z, \calT_Z(-\log D))\cong H^1(Y,\calT_Y)^{\mu_d}$. This is isomorphic to $\Hom(H^{n,0}(Y), H^{n-1,1}_\chi(Y,\CC))$. 

  Finally, the stabilizer of $D$ under $\Aut^\circ(Z)$ is discrete if and only if $H^0(Z, \calT_Z(-\log D))=0$. In this case, we obtain $\dim(\mathscr{U}_T)-\dim \Aut(Z)=h^{n-1,1}_\chi(Y)$.
\end{proof}

 \section{Main Calculation}
\label{section: main}
 
In this section, we perform the calculation of the middle-dimensional Betti number $b_n(Y)$ and use it to find all ball-type $T$ for $Z=(\PP^1)^n$ and $\PP^3$. 

\subsection{Hodge structures of ball type}
\label{subsection: HS of ball type}
We first define Hodge structures of ball type.
\begin{defn}
Let $K$ be an imaginary quadratic extension over a real field $F$. A $K$-Hodge structure of ball type is a $K$-vector space $V$ (of dimension $n$) together with a Hermitian form
\begin{equation*}
h\colon  V\times V\to K
\end{equation*}
of signature $(1, n-1)$ and a filtration
\begin{equation*}
  F^1\subset F^0=V_\CC\coloneqq V\otimes_{K}\CC
\end{equation*}
such that $\dim F^1=1$ and $h(x,x)>0$ for any $x\in F^1\backslash \{0\}$.
\end{defn}

\begin{prop} 
\label{proposition: ball type}
Suppose $Z$ is a smooth projective variety of dimension $n$, $d\ge 3$ an integer, $L$ a line bundle on $Z$ with $-K_Z=(d-1)L$ and $\pi\colon Y\to Z$ a degree-$d$ cyclic cover branching along a simple normal crossing divisor $D\in |dL|$. Assume $h_{\chi}^{p,n-p}=0$ for $p\le n-2$, then $H^n_\chi(Y, \QQ[\xi_d])$ together with the form 
\begin{equation*}
h\colon H^n_\chi(Y, \QQ[\xi_d])\times H^n_\chi(Y, \QQ[\xi_d])\to \QQ[\xi_d]
\end{equation*}
with $h(x,y)=(\xi_d-\overline{\xi}_d)^n\int_Y x\overline{y}$
is a $\QQ[\xi_d]$-Hodge structure of ball type. In particular, the period domain $\BB_T$ is a complex hyperbolic ball of dimension $h^{n-1,1}_\chi$.
\end{prop}
\begin{proof}
We have $H^n_\chi(Y, \CC)=H^{n,0}(Y)\oplus H^{n-1,1}_\chi(Y)$. By Hodge-Riemann bilinear relations, $h$ is a Hermitian form such that $h(x,x)>0$ for $x\in H^{n,0}(Y)\backslash\{0\}$, and $h(x,x)<0$ for $x\in H^{n-1,1}_\chi(Y)\backslash\{0\}$. 

As the space of positive lines in $H^n_\chi(Y,\CC)$ with respect to $h$, $\BB_T$ is a complex hyperbolic ball of dimension $h^{n-1,1}_\chi$.
\end{proof}

Next we specify to the case when $Z$ is a product of projective spaces.
\begin{prop}
\label{proposition: stability}
Suppose $Z=\PP^{n_1}\times\cdots\times \PP^{n_t}$ is a product of projective spaces, $(Z, d, L,T)$ satisfies $d\ge 2$, $-K_Z=(d-1)L$ and all members  of $T$ are effective line bundles. Then $(Z,d,L,T)$ satisfies Condition \ref{condition} and a generic points in $\PP_T$ is stable under the action of $\SL(n_1+1)\times\cdots\times\SL(n_t+1)$ on $H^0(Z,L^d)$.
\end{prop}
\begin{proof}
It is clear that $(Z,d,L,T)$ satisfies Condition \ref{condition}. Thus the space $\sU_T$ of normal crossing divisors of type $T$ is nonempty.
We next show the stability of a generic point of $\PP_T$. Consider the type $\widetilde{T}$ with each member $\widetilde{L}\in \widetilde{T}$ has one component being $1$ and other components being $0$. Then $\widetilde{T}$ is a refinement of $T$ and we have $\PP_{\widetilde{T}}\subset \PP_T \subset \PP H^0(Z,L^d)$. It suffices to show a generic element $x$ of $\PP_{\widetilde{T}}$ is stable under $G=\SL(n_1+1)\times\cdots\times\SL(n_t+1)$. We will actually show any $x\in \sU_{\widetilde{T}}$ is stable under $G$.

We first show that $G_x$ is finite. Suppose $L=\calO(a_1, \cdots, a_t)$. There are $d a_i$ members in $\widetilde{T}$ equal to the pullback of $\calO_{\PP^{n_i}}(1)$ to $Z$. By $-K_Z=(d-1)L$ we have $d a_i\ge n_i+2$ for any $1\le i\le t$. So a simple normal crossing divisor of $\PP^{n_i}$ consisting of $d a_i$ different hyperplanes has finite stabilizer group under the action of $\SL(n_i+1)$. Thus $G_x$ is finite.

The rest argument is similar to the proof in \cite[\S 4.2]{mumford1994geometric} of the stability of smooth hypersurfaces of degree at least $3$ in a projective space. We next construct a $G$-invariant homogeneous function on the subcone of $H^0(Z, L^d)$ over $\PP_{\widetilde{T}}$. There are $d a_1$ line bundles in $\widetilde{T}$ equal to $\calO(1,0,\cdots, 0)$. The corresponding sections $D_j$ ($1\le j\le d a_1$) are determined by $da_1$ linear forms $l_j\in H^0(\PP^{n_1},\calO(1))$. The coefficients of these linear forms form an $(n_1+1)\times da_1$ matrix $M_1$. Since $D_j$ form a normal crossing divisor, any $n_1+1$ column vectors of $M_1$ are linearly independent. Now for an $(n_1+1)\times da_1$ matrix, we consider the product of all $(n_1+1)$-minors. We then obtain an $\SL(n_1+1)$-invariant homogeneous function $f_1\colon H^0(\PP^{n_1}, \calO(1))^{d a_1}\to \CC$, such that $f_1(l_1,\cdots, l_{da_1})\ne 0$. We construct $f_2,\cdots, f_t$ in a similar way. For suitable chosen positive integers $\lambda_1,\cdots, \lambda_t$, the product $f=f_1^{\lambda_1} \cdots f_t^{\lambda_t}$ is a homogeneous function on the subcone of $H^0(Z,L^d)$ over $\PP_{\widetilde{T}}$.

Now any elements $x\in\sU_{\widetilde{T}}$ have finite stabilizer, and $\sU_{\widetilde{T}}$ is the complement of the zero locus of $f$ in $\PP_{\widetilde{T}}$. If there exists $x\in \sU_{\widetilde{T}}$ with nonclosed orbit $Gx$, then there exists an orbit in $\sU_{\widetilde{T}}$ with dimension less than $\dim \PP_{\widetilde{T}}$, contradicting to the finiteness of  stabilizer. Hence the action of $G$ on $\sU_{\widetilde{T}}$ only has closed orbit. Therefore, all elements in $\sU_{\widetilde{T}}$ are stable and we conclude that a generic element in $\sU_T$ is stable.
\end{proof}

By Proposition \ref{proposition: stability}, when $Z$ is a product of projective spaces and $(Z,d,L,T)$ satisfies Condition \ref{condition}, the last equality of Theorem \ref{theorem: local torelli} holds. Namely, $\dim(\sU_T)-\dim \Aut(Z)=h_\chi^{n-1,1}$. The left side of this equality is usually easy to calculate. So we are able to know the value of $h_\chi^{n-1,1}$. To apply Proposition \ref{proposition: ball type}, we need to calculate $b_{n,\chi}=\dim H^n_\chi(Y, \CC)$. This can be reduced to the calculation of the Euler characteristic of $Y$.

\subsection{Betti Numbers}
Denote by $e$ the Euler characteristic function. Let $(Z,d,L,T)$ be a tuple satisfying Condition \ref{condition}. Here $T=(L_1,\cdots, L_m)$. Let $D=D_1+\cdots+D_m$ be a simple normal crossing divisor of type $T$. By Hurwitz formula we have 
\begin{equation}
\label{equation: hurwitz}
e(Y)=de(Z)-(d-1)e(D).
\end{equation} 
On the right hand side of \eqref{equation: hurwitz}, we have $e(Z)=\int_Z c_n(Z)$, where $c_*(Z)$ represent for Chern classes of the tangent bundle of $Z$. The Euler characteristic $e(D)$ is given by 
\begin{equation*}
e(D)=\sum_{i=1}^m e(D_i)-\sum_{1\leq i<j \leq m} e(D_i\cap D_j)+\sum_{1\leq i<j<l\leq m} e(D_i\cap D_j\cap D_j)-\cdots.
\end{equation*} 
So $e(Y)$ can be represented as a linear combination of Chern numbers of $Z$. More precisely, assume the total Chern class of $Z$ is $c(Z)=1+c_1(Z)+c_2(Z)+\cdots +c_n(Z)$ and the first Chern class of $L_i$ is $c_1(L_i)$, then we have:
\begin{prop}
\label{proposition: euler Y}
The Euler characteristic of $Y$ is given by  
\begin{equation*}
e(Y)=e(Z)+(d-1)\int_Z {c(Z) \over (1+c_1(L_1)) (1+c_1(L_2))\cdots (1+c_1(L_m))}
 \end{equation*} 
\end{prop}
\begin{proof}
First we calculate the Euler characteristic of $I=D_{i_1}\cap \cdots \cap D_{i_l}$. We have exact sequence
\begin{equation*}
1\longrightarrow T_I\longrightarrow T_Z\big{|}_I\longrightarrow (\oplus_{j=1}^l L_{i_j})\big{|}_I\longrightarrow 1,
\end{equation*}
and hence relation among total Chern classes
\begin{equation*}
c(T_I)=\frac{c(T_Z\big{|}_I)}{c[(\oplus_{j=1}^l L_{i_j})\big{|}_I]}.
\end{equation*}
So we have 
\begin{equation*}
e(I)=\int_I c(T_I)=\int_Z c(Z)\frac{c_1(L_{i_1})}{1+c_1(L_{i_1})}\cdot\cdots\cdot \frac{c_1(L_{i_l})}{1+c_1(L_{i_l})}.
\end{equation*}
The Euler characteristic of $D$ is given by 
\begin{eqnarray*}
e(D)&=&\sum_{l=1}^k\sum_{1\leq i_1<\cdots<i_l\leq k}(-1)^{l-1}e(D_{i_1}\cap \cdots \cap D_{i_l})\\
&=&\sum_{l=1}^k\sum_{1\leq i_1<\cdots<i_l\leq k}(-1)^{l-1} \int_Z c(Z)\frac{c_1(L_{i_1})}{1+c_1(L_{i_1})}\cdot\cdots\cdot \frac{c_1(L_{i_l})}{1+c_1(L_{i_l})}\\
&=&-\int_Z c(Z)\sum_{l=1}^k\sum_{1\leq i_1<\cdots<i_l\leq k} \frac{-c_1(L_{i_1})}{1+c_1(L_{i_1})}\cdot\cdots\cdot \frac{-c_1(L_{i_l})}{1+c_1(L_{i_l})}\\
&=&-\int_Z {c(Z)  ({1\over 1+c_1(L_1)}{1\over 1+c_1(L_2)}\cdots {1\over 1+c_1(L_m)}-1)}\\
&=&e(Z)-\int_Z {c(Z) \over (1+c_1(L_1)) (1+c_1(L_2))\cdots (1+c_1(L_m))}.
\end{eqnarray*}
So the Euler characteristic of $Y$ is 
 \begin{align*}
e(Y) &=de(Z)-(d-1)e(D)\\
&=e(Z)+(d-1)\int_Z {c(Z) \over (1+c_1(L_1)) (1+c_1(L_2))\cdots (1+c_1(L_m))}
\end{align*}
and we conclude the proposition.
\end{proof}

We would like to focus on the Hodge structures on $H^n_{\chi^i}$ for nontrivial characters $\chi^i$. By \cite[Corollary 1.3]{arapura2014hodge}, the invariant cohomology $H^*(Y,\CC)^{\mu_d}$ is isomorphic to $H^*(Z,\CC)$. If moreover $Z$ is a Fano variety (equivalently, $L$ is ample), a version of weak Lefschetz theorem holds, namely $H^i(Y,\QQ)\cong H^i(Z,\QQ)$ for $i\ne n$, see \cite[Corollary 1.6]{arapura2014hodge}. So we consider 
\begin{equation*}
b'_n(Y)=\sum\limits_{i=1}^{d-1} \dim H^n_{\chi^i}(Y,\CC) =b_n(Y)-b_n(Z)=(-1)^n(e(Y)-e(Z))
\end{equation*}
which is called the primitive middle Betti number of $Y$. A direct corollary of Proposition \ref{proposition: euler Y} is the following. 

\begin{cor}
\label{Corollary: primitive betti number}
The primitive betti number of $Y$ is given by 
\begin{equation}
\label{equation: middle betti}
b_n^\prime(Y)=(-1)^{n} (d-1)\int_Z {c(Z) \over (1+c_1(L_1)) (1+c_1(L_2))\cdots (1+c_1(L_m))}.
\end{equation}
\end{cor}

\subsection{Case $Z=(\PP^1)^n$}
In this section we consider the case $(Z,d,L)=((\PP^1)^n,3,\calO(1)^{\boxtimes n})$. Suppose $T=(L_1,\cdots,L_m)$ is a partition of $3L$ such that $(Z,d,L,T)$ satisfies Condition \ref{condition}. Let $h_i\in H^2(Z, \ZZ)$ be the first Chern class of the pullback of $\calO(1)$ from the $i$-th $\PP^1$-component to $Z$. The first Chern class of the line bundle $L_j$ is denoted by $\alpha_j=\sum\limits_{i=1}^{n} a_j^i h_i$. The total Chern class of $Z$ is $(1+2h_1)\cdots(1+2h_n)=(1+h_1)^2\cdots (1+h_n)^2$. 

We define the following function of variables $h=(h_1, \cdots, h_n)$:
\begin{equation*}
  f(h_1, \cdots, h_n)=\frac{(1+h_1)^2\cdots (1+h_n)^2}{(1+\alpha_1)\cdots (1+\alpha_m)}. 
\end{equation*}

Then Corollary \ref{Corollary: primitive betti number} implies that 
\begin{equation}
\label{equation: prime betti}
    b_n^\prime(Y)=(-1)^{n}\cdot 2 (\partial_{h_1}\cdots \partial_{h_n} f(h_1, \cdots, h_n)\big{|}_{h=0}).
\end{equation}

We can use log derivative method to solve this. In order to make notation easier, we use $\partial_i$ instead of $\partial_{h_i}$ for short. Given a positive integer $k$, let a dividing of the set $\{1,\cdots, k\}$ be $\{A_1, \cdots, A_l\}$ such that $\{1,\cdots, k\}=A_1\sqcup\cdots\sqcup A_l$. Denote by $\Pi_k$ the set of dividings of $\{1,\cdots, k\}$. For $I\in \Pi_k$, we define 
\begin{equation*}
\partial_I f\coloneqq \prod_{A\in I} [(\prod_{i\in A} \partial_i) f].
\end{equation*}
The following lemma is straightforward to prove by induction on $k$.
\begin{lem}
\label{lemma: derivative log}
For each $k$, we have
\begin{equation*}
\partial_1\cdots\partial_k f=f\sum_{I\in\Pi_k} \partial_I(\log f).
\end{equation*}
\end{lem}
\begin{proof}
We induct on $k$. The case $k=1$ is clear. Assume $k\ge 2$. Then
\begin{eqnarray*}
\partial_1\cdots\partial_k f &=& \partial_k(f\sum_{I'\in\Pi_{k-1}}\partial_{I'}\log f)\\
&=& \sum_{I'\in\Pi_{k-1}} (\partial_k f) (\partial_{I'} \log f)+f\sum_{I'\in \Pi_{k-1}}\partial_k(\partial_{I'}\log f)\\
&=& f\sum_{I'\in\Pi_{k-1}} (\partial_k \log f) (\partial_{I'} \log f)+f\sum_{I\in \Pi_k\text{ with } \{k\}\notin I} \partial_I\log f\\
&=& f\sum_{I\in\Pi_k} \partial_I(\log f)
\end{eqnarray*}
and the lemma follows.
\end{proof}

We have $\sum\limits_{j=1}^m a_j^i=3$ for each $i\in\{1,\cdots,n\}$. We use Lemma \ref{lemma: derivative log} to calculate $\partial_1\cdots\partial_n f\big{|}_{h=0}$. For $I\in\Pi_k$, we denote $\epsilon_I=\prod\limits_{A\in I}(|A|-1)!$. 
\begin{prop}
\label{proposition: key calculation}
We have
\begin{equation*}
\partial_1\cdots\partial_n f\big{|}_{h=0}=(-1)^n \sum_{I\in\Pi_n} [\epsilon_I \prod_{A\in I, |A|\ge 2}(\sum_{j=1}^m \prod_{i\in A}a_j^i)].
\end{equation*}
\end{prop}
\begin{proof}
We have 
\begin{equation*}
\log f=2\sum_i\log(1+h_i)-\sum_j\log(1+\alpha_j). 
\end{equation*}
For each $i\in\{1,\cdots, n\}$, we have 
\begin{equation*}
\partial_i\log f=2(1+h_i)^{-1}-\sum_j a_j^i(1+\alpha_j)^{-1}. 
\end{equation*}
Take $h=0$, we have 
\begin{equation}
\label{equation: derivative once}
\partial_i\log f\big{|}_{h=0}=2-\sum_j a_j^i=-1.
\end{equation}
For a subset $A\subset\{1,\cdots,n\}$ with $|A|\ge 2$, we have
\begin{eqnarray*}
(\prod_{i\in A} \partial_i) \log f &=& -\sum_j (\prod_{i\in A} \partial_i) \log(1+\alpha_j)\\
&=& \sum_j [(-1)^{|A|}(|A|-1)! (\prod_{i\in A}a_j^i) (1+\alpha_j)^{-|A|}]
\end{eqnarray*}
which implies that
\begin{equation}
\label{equation: derivative multiple}
(\prod_{i\in A} \partial_i) \log f \big{|}_{h=0}=\sum_j [(-1)^{|A|}(|A|-1)!\prod_{i\in A}a_j^i]
\end{equation}
For any $I\in\Pi_n$, combining \eqref{equation: derivative once} and \eqref{equation: derivative multiple}, we obtain
\begin{equation*}
\partial_I \log f\big{|}_{h=0}=(-1)^n\epsilon_n \prod_{A\in I,|A|\ge 2} (\sum_j \prod_{i\in A} a_j^i).
\end{equation*}
Then the proposition follows from Lemma \ref{lemma: derivative log}.
\end{proof}
Combining \eqref{equation: prime betti} and Proposition \ref{proposition: key calculation}, we obtain
\begin{equation}
\label{equation: h chi}
h_{\chi}^n(Y)=h_{\overline{\chi}}^n(Y)={1\over 2}b'_n(Y)=\sum_{I\in\Pi_n} [\epsilon_I \prod_{A\in I, |A|\ge 2}(\sum_{j=1}^m \prod_{i\in A}a_j^i)].
\end{equation}

We can calculate $h_\chi^{n-1,1}(Y)$ using the GIT moduli space of simple normal crossing divisors.
\begin{lem}
\label{lemma: git dim}
The GIT moduli space of simple normal crossing divisors of type $(L_1, \cdots, L_m)$ has dimension 
\begin{equation*}
\sum_j(\prod_i (a_j^i+1)-1)-3n.
\end{equation*}
and this is also the value of $h_\chi^{n-1,1}(Y)$.
\end{lem}
\begin{proof}
The calculation of the GIT dimension is straightforward. By Theorem \ref{theorem: local torelli} and Proposition \ref{proposition: stability} (and the discussion after it), the dimension of the GIT moduli space equals to $h_\chi^{n-1,1}(Y)$. 
\end{proof}

Notice that $b_{n,\chi}(Y)=\sum_{1\le p\le n}h^{p,n-p}_{\chi}(Y)$. Then from \eqref{equation: h chi} and  Lemma \ref{lemma: git dim}, we have the following result.
\begin{prop}
\label{proposition: difference}
For $(Z, d)=((\PP^1)^n, 3)$ and $n\ge 3$, we have
\begin{equation*}
\sum_{1\le p\le n-2} h_{\chi}^{p,n-p}(Y)=\sum_{I\in\Pi_n, |I|<n} [(\epsilon_I-\delta_I)\prod_{A\in I} (\sum_j \prod_{i\in A} a_j^i)].
\end{equation*}
Here, $\delta_I$ equals to $1$ if there is exactly one element $A\in I$ with $|A|\ge 2$, and equals to $0$ otherwise. 
\end{prop}

\begin{rmk}
When $Z=(\PP^1)^2$, $d=3$ and $D\in |\calO(3,3)|$, the Calabi--Yau orbifold $Y$ is a singular $K3$ surface. By \eqref{equation: prime betti} and Proposition \ref{proposition: key calculation} we have $b_2^\prime(Y)=2(-5+\sum_j((1+a_j^{1})(1+a_j^{2})-1))$. On the other hand, by Lemma \ref{lemma: git dim}, $h^{1,1}_\chi= -6+\sum_j ((1+a_j^{1})(1+a_j^{2})-1)$. These two calculations match with the equality $2h_{\chi}^{1,1}(Y)=b'_2(Y)-1$, which is a priori known.
\end{rmk}

\begin{thm}
\label{theorem: main1}
Let $(Z, d, L)=((\PP^1)^n, 3, \calO(1)^{\boxtimes n})$. For a partition type $T=(L_1,\cdots, L_m)$ of $3L$, the following two statements are equivalent:
\begin{enumerate}
\item $\sum\limits_{0\le p\le n-2} h_{\chi}^{p,n-p}(Y)=0$,
\item each $\alpha_j=c_1(L_j)$ has at most two nonzero components and each $\alpha_{j_1}+\alpha_{j_2}$ has at most three nonzero components.
\end{enumerate}
If these hold, then the period domain $\BB_T$ is a complex hyperbolic ball.
\end{thm}
\begin{proof}
According to Proposition \ref{proposition: difference}, the equality $\sum\limits_{0\le p\le n-2} h_{\chi}^{p,n-p}(Y)=0$ holds if and only if
\begin{equation}
\label{equation: vanishing}
\prod_{A\in I} (\sum_j \prod_{i\in A} a_j^i)=0
\end{equation}
for all $I\in\Pi_n$ with $|I|<n$ and $\epsilon_I>\delta_I$. This condition on $I$ is equivalent to $|I|\le n-2$. For $I$ of type $(1,\cdots, 1, 3)$, \eqref{equation: vanishing} holds if and only if for each $j\in\{1,\cdots,m\}$, there are at most two nonzero numbers among $a_j^1,\cdots, a_j^n$. For $I$ of type $(1,\cdots,1,2,2)$, \eqref{equation: vanishing} hold if and only if for any $j_1, j_2\in\{1,\cdots,m\}$, we cannot find four distinct numbers $i_1, i_2, i_3, i_4$ among $\{1,\cdots,n\}$, such that $a_{j_1}^{i_1}a_{j_1}^{i_2}a_{j_2}^{i_3}a_{j_2}^{i_4}\ne 0$. It is then straightforward to see that $\BB_T$ is ball if and only if each $\alpha_j$ has at most two nonzero components and each $\alpha_{j_1}+\alpha_{j_2}$ has at most three nonzero components.
\end{proof}

Theorem \ref{theorem: main} is a straightforward conclusion from Theorem \ref{theorem: main1}.
\begin{proof}[Proof of Theorem \ref{theorem: main}]
The case $n\le 4$ is clear. 

Assume $n\ge 5$ and $T=(\alpha_1, \cdots, \alpha_m)$ is a partition such that each $\alpha_j$ has at most two nonzero components, and each $\alpha_{j_1}+\alpha_{j_2}$ has at most three nonzero components. Suppose to the contrary that $T$ is not a Fermat-type half-twist of a ball-type partition for $n-1$. Then for each $t=1,\cdots,n$, there exsits $k\ne t$, $1\le k\le n$, such that certain $\alpha_j$ has nonzero components at both positions $t$ and $k$. Without loss of generality, we may assume that $\alpha_1$ has nonzero components at positions $1$ and $2$. All $\alpha_j$ ($j\ge 2$) have at most one nonzero component for positions $3,4,\cdots,n$. For each $t\in\{3,4,\cdots,n\}$, there exists $\alpha_{i_t}$ with nonzero component at position $t$. There is a unique $\epsilon_t\in\{1,2\}$ such that $\alpha_{i_t}$ has nonzero components at positions $\epsilon_t, t$. The numbers $\epsilon_3, \cdots, \epsilon_n$ have to be the same. We assume they all equal to $1$. Considering the sum of the first components of $\alpha_1$ and $\alpha_{i_t}$, we have $3\ge 1+(n-2)=n-1$, which contradicts to $n\ge 5$.

The dimension of the balls in the four maximal cases for $n=3,4$ can be directly calculated via the dimension of the GIT moduli spaces of normal crossing divisors.
\end{proof}

When $n=2$, the cyclic cover $Y$ associated with any partitions are $K3$ surfaces with ADE singularities and non-symplectic automorphism of order $3$. The period domains for those $K3$ surfaces are complex hyperbolic balls of dimension at most $9$. In \cite{yu2024commensurability} we used these examples to study commensurability relation among Deligne--Mostow ball quotients.

When $n\ge 3$, we obtain Calabi--Yau manifolds such that the period maps factor through arithmetic ball quotients. We will show arithmeticity in \S\ref{section: period map} and discuss their relations to Deligne--Mostow theory in \S\ref{section: relation to DM}. 

For $n=3$, there are totally $40$ ball-type partitions, which are listed in Table \ref{table: CY Deligne--Mostow}. For $n=4$, there are $10$ refinements of $T=(\calO(1,3,0,0), \calO(1,0,3,0), \calO(1,0,0,3))$, see Table \ref{table: dimension 4}.

\subsection{Half-twist Construction}
\label{subsection: half-twist}
Recall that in Definition \ref{definition: fermat half twist} we defined the Fermat-type half-twist partition type $\widetilde{T}$ of a partion $T$ of $3L$ on $Z$. In this section we study the Hodge structure of the triple cyclic cover $\widetilde{Y}\to Z\times \PP^1$ branched along a normal crossing divisor of type $\widetilde{T}$.

Suppose $Z_1, Z_2$ are two smooth projective varieties. Assume $L_i$ is a line bundle on $Z_i$ with a nonzero section $s_i\in H^0(Z_i, L_i^3)$. Then $s_i$ defines a triple cover $Y_i\to Z_i$. Explicitly, $Y_i=\Spec(\calO_{Z_i}\oplus L_i^{-1}\oplus L_i^{-2})\to Z_i$ where the algebra structure for the sheaf $\calO_{Z_i}\oplus L_i^{-1}\oplus L_i^{-2}$ is given by $L_i^{-3}\to \calO_{Z_i}, x\mapsto x s_i$ for any section $x$ of $L_i^{-3}$. Let $\rho_i$ be the automorphism of $Y_i$ induced by multiplying $\omega=\exp({2\pi\sqrt{-1}\over 3})$ on $L_i$.

\begin{prop}
\label{proposition: diagonal quotient}
The quotient $(Y_1\times Y_2)/\langle \rho_1\times \rho_2^{-1} \rangle$ is isomorphic to the triple cover of $Z_1\times Z_2$ given by $(p_1^* s_1)(p_2^* s_2)\in H^0(Z_1\times Z_2, (L_1\boxtimes L_2)^3)$, where $p_i$ is the projection of $Z_1\times Z_2$ to $Z_i$.
\end{prop}
\begin{proof}
This is because that the invariant part
\begin{equation*}
[(\calO_{Z_1}\oplus L_1^{-1}\oplus L_1^{-2})\otimes (\calO_{Z_2}\oplus L_2^{-1}\oplus L_2^{-2})]^{\rho_1\times \rho_2^{-1}}
\end{equation*}
equals to $\calO_{Z_1\times Z_2}\oplus (L_1\boxtimes L_2)^{-1} \oplus (L_1\boxtimes L_2)^{-2}$.
\end{proof}

In our situation, we take $(Z_1, L_1, s_1)=(Z,L,s)$ such that $s$ defines a normal crossing divisor of type $T$ on $Z$ and the associated triple cover $Y_1=Y$ is a Calabi--Yau orbifold. We take $(Z_2, L_2)=(\PP^1, \calO(1))$ and let $Y_2=E$ be a triple cover of $\PP^1$ branched at three points determined by $s_2\in H^0(\PP^1, \calO(3))$.  Then the divisor type determined by the section $(p_1^* s_1)(p_2^* s_2)$ is the Fermat-type half-twist of $T$ (see Definition \ref{definition: fermat half twist}). By Proposition \ref{proposition: diagonal quotient}, the triple cover $\widetilde{Y}\to Z\times \PP^1$ given by $(p_1^* s_1)(p_2^* s_2)$ is isomorphic to the quotient of $Y\times E$ by $\rho_1\times \rho_2^{-1}$. From this geometric description, we have the following Hodge-theoretic result:
\begin{prop}
\label{proposition: half twist ball}
Suppose $(Z, d=3, L, T)$ satisifies Condition \ref{condition} and is of ball type. Let $\widetilde{T}$ be the Fermat-type half-twist of $T$. Then $(Z\times \PP^1, d=3, L\boxtimes \calO_{\PP^1}(1), \widetilde{T})$ also satisfies Condition \ref{condition} and is of ball type.
\end{prop}
\begin{proof}
Let $n$ be the dimension of $Z$. We have $\widetilde{Y}=(Y\times E)/\langle \rho_1\times \rho_2^{-1}\rangle$. Thus
\begin{equation*}
H^{n+1}(\widetilde{Y}, \QQ[\omega])=(H^n(Y,\QQ[\omega])\otimes H^1(E, \QQ[\omega]))^{\rho_1\times\rho_2^{-1}}
\end{equation*}
Let $\chi_i$ be the character for $\langle\rho_i\rangle$ such that $H_{\chi_i}(Y_i,\CC)$ contains $H^{n_i,0}(Y_i)$. Here $Y_1=Y, Y_2=E, n_1=n, n_2=1$. Then 
\begin{equation*}
H^{n+1}(\widetilde{Y}, \QQ[\omega])=[H^n_{\chi_1}(Y, \QQ[\omega])\otimes H^1_{\chi_2}(E,\QQ[\omega])]\oplus [H^n_{\overline{\chi}_1}(Y, \QQ[\omega])\otimes H^1_{\overline{\chi}_2}(E,\QQ[\omega])]
\end{equation*}
The deck transformation group for $\widetilde{Y}\to Z\times \PP^1$ is generated by $\rho_1\times id$. It has a character $\widetilde{\chi}$ such that $H^{n+1}_{\widetilde{\chi}}(Y,\QQ[\omega])=H^n_{\chi_1}(Y, \QQ[\omega])\otimes H^1_{\chi_2}(E,\QQ[\omega])$ contains $H^{n+1,0}(Y)$. Since $H^1_{\chi_2}(E,\QQ[\omega])=\QQ[\omega]$, we have an isomorphism $H^{n+1}_{\widetilde{\chi}}(Y,\QQ[\omega])\cong H^n_{\chi_1}(Y, \QQ[\omega])$ of Hodge structures of ball type.
\end{proof}

\subsection{Case $Z=\PP^3$}
\label{section: case Pn}

In this section we consider the case $Z=\PP^3$ with $d=3$ or $5$. We will revisit Example \ref{example: sheng}. First we consider more generally $Z=\PP^n$. Let $s$ be the total degree of $D_i$. To make $Y$ a Calabi--Yau orbifold, we need ${d\over d-1}={s\over n+1}$. Let $l_i$ be the degree of $D_i$ for $1\le i\le m$.

Denote $s_t=\sum\limits_{i=1}^m l_i^t$. Let $\pi$ be a partition of $n$ into non-ordered positive integers. For $t\in\pi$, denote by $m(t,\pi)$ the multiplicity of $t$ in $\pi$. From Formula \eqref{equation: middle betti}, we deduce:
\begin{prop}
\label{proposition: middle betti Pn}
The primitive middle betti number of $Y$ is 
\begin{equation}
\label{equation: middle betti Pn}
b_n'(Y)=(d-1)\sum_{\pi} \prod_{t\in\pi} {s_t-n-1\over m(t,\pi)!t}
\end{equation}
\end{prop}
\begin{proof}
We have 
\begin{equation*}
b_n'(Y)=(-1)^n (d-1) \int_Z {(1+h)^{n+1}\over (1+l_1 h)\cdots(1+l_m h)}
\end{equation*}
Let $f(h)={(1+h)^{n+1}\over (1+l_1 h)\cdots(1+l_m h)}$, we need to calculate $f^{(n)}(h)\big{|}_{h=0}$. Let $g(h)=\ln f(h)$. Then 
\begin{align*}
f'&=g' e^g, \\
f''&=(g''+(g')^2)e^g, \\
&\cdots \\
f^{(n)}&=n!e^g (\sum_\pi \prod_{t\in\pi} {g^{(t)} \over m(t,\pi)!t!})
\end{align*}
We have $g^{(t)}\big{|}_{h=0}=(-1)^{t-1}(t-1)!(n+1-s_t)$. Then by straightforward calculation we have Formula \eqref{equation: middle betti Pn}.
\end{proof}

\begin{cor}
For $Z=\PP^3$ and $d=3$, the period domain is of ball type if and only if $l_1=\cdots=l_6=1$. In this case the ball has dimension $3$.
\end{cor}
\begin{proof}
By Formula \eqref{equation: middle betti Pn}, we have $b_n'(Y)={2\over 3}s_3+2s_2-8$.

The dimension of the GIT moduli space of hypersurfaces of degree $l_1, \cdots, l_m$ in $\PP^3$ is 
\begin{equation*}
\sum_{i=1}^m[\binom{n+l_i}{n}-1]-(n+1)^2+1={1\over 6} s_3+s_2-4.
\end{equation*}
Thus the period domain is of ball type if and only if $s_3=6$, which only happens when $l_1=\cdots=l_6=1$.
\end{proof}

\begin{cor}
For $Z=\PP^3$ and $d=5$, the period domain is of ball type if and only if $(l_1, \cdots, l_m)=(1,1,1,1,1), (1,1,1,2)$ or $(1,2,2)$. The corresponding balls have dimension $0,3,6$ respectively.
\end{cor}
\begin{proof}
By Formula \eqref{equation: middle betti Pn}, we have $b_n'(Y)=4({1\over 3}s_3+{1\over 2}s_2-{19\over 6})$. The dimension of the GIT moduli space of hypersurfaces of degree $l_1, \cdots, l_m$ is ${1\over 6}s_3+s_2-{35\over 6}$.

We have $({1\over 3}s_3+{1\over 2}s_2-{19\over 6})-({1\over 6}s_3+s_2-{35\over 6})=1+{1\over 6}\sum\limits_{i=1}^m l_i(l_i-1)(l_i-2)$. Thus the period domain is of ball type if and only if all $l_i$ equals to $1$ or $2$.
\end{proof}

\section{Refinement Relation}
	\label{section: refinement, incidence and complete}
	The set of partitions of $dL\in\Pic(Z)$ admits a natural poset structure induced by taking refinements. In this section we study the relation among period maps under refinements. In particular, we prove that the refinement of a partition of ball type is still of ball type. For a semistable degeneration of projective varieties, Clemens-Schmid sequence relates the limiting mixed Hodge structure with the mixed Hodge structure of special fiber. In our case, the degeneration is not semistable, and we apply a generalization of Clemens-Schmid sequence obtained by Kerr and Laza \cite{kerr2021hodge}.
	
Let $(Z, d, L, L_1, \cdots, L_m)$ be a tuple satisfying Condition \ref{condition}. We first consider the refinement of one component. Let  $T'=(M_0, M_1, L_2, L_3, \cdots, L_m)$ be a refinement of $T=(L_1, \cdots, L_m)$  such that $L_1=M_0+M_1$ and $(Z,d,L,T')$ also satisfies Condition \ref{condition}. Choose generic sections $f_i\in H^0(Z,L_i)$ and $g_j\in H^0(Z,M_j)$ such that $f_1\cdots f_m$ and $f_2 \cdots f_m g_0 g_1$ define simple normal crossing divisors in $Z$.
	
Let $\mathscr{Y}$ be the family of Calabi--Yau orbifolds over $\Delta=\{t\big{|}|t|<1\}$, such that the fiber $Y_t$ over $t$ is defined by
\begin{equation}
\label{equation: refinement}
y^d=(tf_1+g_0 g_1)f_2\cdots f_m.
\end{equation}
Then $Y_0$ is the degree-$d$ cyclic cover of $Z$ branched along a divisor of type $T'$. By rescaling $f_1$, we can assume that for $t\in \Delta-\{0\}$, $Y_t$ is the degree-$d$ cyclic cover of $Z$ branched along a simple normal crossing divisor of type $T$. Hence we obtain an equisingular family of Calabi--Yau varieties over $\Delta-\{0\}$. This is the same setting as \cite[Diagram (4.6)]{kerr2021hodge}. Let $M$ be the monodromy operator on $H^*(Y_t, \QQ)$ for $t\ne 0$. Then we have the following proposition.
		
	\begin{prop}
		\label{proposition: Clemens-Schmid sequence for refinement}
		Let $M$ be the monodromy operator on cohomology group $H^n_\chi(Y_t)$ with $H^n_{\chi, \lim}(Y_t)^M$ the fixed part. Then the morphism
		\begin{equation*}
			H^n_\chi(Y_0)\to H^n_{\chi,\lim}(Y_t)^M
		\end{equation*}
is an isomorphism.
	\end{prop}
\begin{proof}
If $m=1$, then total space $\mathscr{Y}$ is smooth and we can apply \cite[Sequence (5.11), Theorem 5]{kerr2021hodge} to obtain the proposition. For general $m$, similar argument still works as follows. Applying nearby and vanishing cycles formalism to the intersection complex $\IC^{\bullet}(\mathscr{Y})$, we also obtain the exact sequence as \cite[Diagram (5.7)]{kerr2021hodge}.
\begin{equation}
	\label{sequence: Clemens-Schmid}
	0 \rightarrow \mathrm{IH}_{\mathrm{lim}}^{n-2}\left(Y_{t}\right)^M(-1) \stackrel{\mathrm{sp}}{\rightarrow} \mathrm{IH}_{c}^{n}\left(\mathcal{Y}\right) \rightarrow \mathrm{IH}^{n}\left(\mathcal{Y}\right) \stackrel{\mathrm{sp}}{\rightarrow} \mathrm{IH}_{\mathrm{lim}}^{n}\left(Y_{t}\right)^{M} \rightarrow 0
\end{equation}
From the construction, we know that $\mathscr{D}=\{(tf_1+g_0 g_1)f_2\cdots f_m=0\}$ is a simple normal crossing divisor on $Z\times \Delta$, and $\mathscr{Y}$ is the cyclic cover of $Z\times \Delta$ branched along $\mathscr{D}$. Thus $\mathscr{Y}$ has only finite quotient singularities \cite[Lemma 1.1]{arapura2014hodge}, and hence is a $\QQ$-homology manifold, see \cite[Proposition 1.4]{steenbrink1977mixed}. Thus, there is an isomorphism $\IC^{\bullet}(\mathscr{Y})\cong \QQ_{\mathscr{Y}}[n+1]$.
By \cite[Chapter III, 8.13a]{gelfand2002methods}, we have $H^*(\mathscr{Y})\cong H^*(Y_0)$. Since the family admits an equivariant cyclic group operation, the sequence \eqref{sequence: Clemens-Schmid} gives an exact sequence
$$
H_{n+2,\chi} (Y_0)(-n-1) \to H^n_\chi(Y_0)\to H^n_{\lim,\chi}(Y_t)^M \to 0
$$
Since $H_{n+2,\chi} (Y_0)=0$, the specialization map $H^n_\chi(Y_0)\to H^n_{\lim,\chi}(Y_t)^M$ is an isomorphism.
\end{proof}

Next, we focus on degeneration of ball-type partitions.

\begin{prop}
Assume that the period domain $\BB_T$ associated with $T=(L_1, \cdots, L_m)$ is a complex hyperbolic ball. Then the period domain associated with $T'=(M_0, M_1, L_2, L_3, \cdots, L_m)$ is also a complex hyperbolic ball.
\end{prop}

\begin{proof}
Since $Y_t$ has ball-type period domain, we have $h^{n-k,k}_\chi(Y_t)=0$ for all $k\geq 2$. Therefore, we have $h^{n-k,k}_\chi(Y_0)=0$ for all $k\geq 2$ by Proposition \ref{proposition: Clemens-Schmid sequence for refinement}. So the period domain $\BB_{T'}$ is also a complex hyperbolic ball.
\end{proof}

To better understand the local behavior, we prove the following result about the monodromy. 
\begin{prop}
\label{proposition: monodromy finite order}
Assume the period domain $\BB_T$ associated with $T=(L_1, \cdots, L_m)$ is a complex hyperbolic ball, then the monodromy operator $M$ defined above on $H^n_\chi(Y_t)$ has finite order.
\end{prop}
\begin{proof}
Let $M=M^{un} M^{ss}$ be the Jordan decomposition of the operator $M$ into unipotent and semi-simple parts. Let $N=\log M^{un}$ be the nilpotent operator. Then $N$ commutes with the action of $\mu_d$, hence $N$ acts on $H^n_\chi(Y_t)$ and defines a weight filtration 
\begin{equation*}
H^n_\chi(Y_t)=W_{2n}\supset \cdots\supset W_0=0.
\end{equation*}
Since $\BB_T$ is of ball type, the Hodge filtration on $H^n_\chi(Y_t)$ is
\begin{equation*}
H^{n,0}(Y_t)=F^n\subset F^{n-1}=H^n_\chi(Y_t)
\end{equation*}
the Hodge filtration on $H^n_\chi(Y_t)$. Proposition \ref{proposition: Clemens-Schmid sequence for refinement} implies that $H^{n,0}(Y_0)\cong (H^{n,0}(Y_t))^M$ has dimension one. So $H^{n,0}(Y_t)$ is invariant under the action of $M$. Thus $N$ is zero on $H^{n,0}(Y_t)$.

We next show $N=0$, which is equivalent to show $W_n=H^n_\chi(Y_t)$. Otherwise there exists $r\in\{1,2,\cdots,n\}$, such that $\Gr^N_{n+r}=W_{n+r}/W_{n+r-1}\ne 0$. Suppose $x\in W_{n+r}$ such that $[x]\in \Gr^N_{n+r}$ is nonzero. We may assume either $[x]\in H^{n,r}$ or $[x]\in H^{n-1,r+1}$. There is an isomorphism $N^r\colon \Gr^N_{n+r}\cong \Gr^N_{n-r}[-r]$ of pure Hodge structures. If $[x]\in H^{n,r}$, then $x\in F^n$, which implies that $N^r([x])=0$, contradiction. If $[x]\in H^{n-1,r+1}$, then $N^r([x])\in H^{n-r-1,1}$. Since $F^{n-r-1}=F^{n-r}$, we must have $H^{n-r-1,1}=0$, again contradiction. 

Thus $N=0$ and $M=M^{ss}$. By Borel's Monodromy theorem (\cite[Lemma 4.5, Theorem 6.1]{schmid1973variation}), $M=M^{ss}$ has finite order.
\end{proof}
We summarize the above results  in the following theorem.
	
\begin{thm}
\label{theorem: refinement}
Let $(Z, d, L, T)$ be a tuple satisfying Condition \ref{condition} and has period domain $\BB_T$ of ball type. Assume $T=(L_1, \cdots, L_m)$ with a refinement $T'=(L'_1, \cdots, L'_k)$ such that $(Z,d,L,T')$ also satisfies Condition \ref{condition}. Then the period domain $\BB_{T'}$ is naturally a totally geodesic sub-ball of $\BB_T$.
\end{thm}
\begin{proof}
Using induction, we only need to consider the case of refinement of one line bundle as \eqref{equation: refinement}. We have a finite-order monodromy operator $M$ and an isomorphism between Hodge structures $H^n_\chi (Y_0)\cong H^n_{\lim,\chi}(Y_t)^M$. The period domain $\BB_{T'}$ associated with $T'$ is simply the fixed locus of $\BB_T$ under the automorphism $M$, hence a totally geodesic sub-ball.		
\end{proof}

\begin{rmk}
We consider the family \eqref{equation: refinement} over $\Delta$. Then the period map maps $\Delta-0$ to $\BB_T/\langle M\rangle$ extends to $\Delta$, with the image of $0$ in the sub-ball $\BB_{T'}$. 
\end{rmk}
	
\section{Completeness of Calabi--Yau Manifolds}
\label{section: completeness}
Let $T$ be ball-type partition in Theorem \ref{theorem: main}. The family $\Y_T\to \sU_T$ is a complete family of Calabi--Yau orbifolds. A natural question is that when is the simultaneous crepant resolution $\X_T\to\sU_T$ a complete family of Calabi--Yau manifolds.

We need the following definition given by Sheng, Xu and Zuo \cite{sheng2013maximal}.
\begin{defn}
\label{definition: good crepant}
For an arbitrary projective Calabi--Yau orbifold $Y$, we call a crepant resolution $X\to Y$ a good one, if there exist successive blowups $f\colon X=X_N\to X_{N-1}\to\cdots\to X_1\to Y$, such that:
\begin{enumerate}[(i)]
\item the composition $f$ is a crepant resolution of $Y$;
\item the map $f^{-1}(Y_0)\to Y_0$ is an isomorphism;
\item each step is a blowup along a smooth center of codimension $2$, with the exceptional divisor a $\PP^1$-bundle over the center.
\end{enumerate}
	\end{defn}
	
In \cite[Corollary 2.6]{sheng2013maximal}, Sheng, Xu and Zuo construct explicitly good crepant resolutions for Calabi--Yau orbifolds $Y$ arising as cyclic covers.
\begin{prop}[Sheng-Xu-Zuo]
\label{proposition: sheng xu zuo}
Suppose $Y$ is a Calabi--Yau orbifold arising as cyclic covers of certain smooth projective variety branching along a normal crossing divisor $D=\sum\limits_{i=1}^m D_i$. Then there exists a good crepant resolution $X=X_N\to X_{N-1}\to\cdots\to X_1\to Y$ for $Y$ with the center of each blowup being the intersection of two divisors arising as a strict transform of $D_i$ or exceptional divisors of previous blowups.
\end{prop}
	
\begin{rmk}
Finding crepant resolutions for Calabi--Yau orbifolds is an important problem. In \cite{cynk2007higher}, Cynk and Hulek investigate the existence of crepant resolutions for many Calabi--Yau orbifolds, which are finite quotients of products of Calabi--Yau manifolds with smaller dimensions.
	\end{rmk}
	
Let $Y$ be a Calabi--Yau orbifold with crepant resolution $f\colon X\to Y$. We have a linear equivalence $K_X\cong f^*K_Y$. Let $X_0\coloneqq f^{-1}Y_0$ be the preimage of the smooth part $Y_0$ of $Y$. Since $Y$ is Calabi--Yau, the restriction of $K_Y$ to $Y_0$ is linearly equivalent to $\calO_{Y_0}$. Thus, the restriction of $K_X$ to $X_0$ is linearly equivalent to $\calO_{X_0}$. However, we do not necessarily conclude that $X$ is a Calabi--Yau manifold. We show that this is the case once the crepant resolution $X\to Y$ is good.
	
\begin{prop}
\label{proposition: crepant is CY}
Let $Y$ be a Calabi--Yau orbifold with good crepant resolution $f\colon X\to Y$. Then the induced map $f^*\colon H^{k,0}(Y)\to H^{k,0}(X)$ is an isomorphism for $k=0,1,\cdots, n$. In particular, $X$ is a Calabi--Yau manifold.
\end{prop}
\begin{proof}
Let $f\colon X=X_N\to X_{N-1}\to\cdots\to X_1\to X_0=Y$ be the decomposition that satisfies the conditions of Definition \ref{definition: good crepant}. Take $i\in\{0,1,\cdots, N-1\}$, consider $f_i\colon X_{i+1}\to X_i$. Denote by $C_i\subset X_i$ the subvariety of codimension $2$, such that $f_i$ is the blowup along $C_i$. The preimage $E_i=f_i^{-1}C_i$ is a $\PP^1$-bundle over $C_i$. By \cite[Corollary-Definition 5.37]{peters2008mixed}, we have a long exact sequence of mixed Hodge structures:
\begin{equation*}
\cdots\to H^k(X_i)\to H^k(X_{i+1})\oplus H^k(C_i)\to H^k(E_i)\to H^{k+1}(X_i)\to \cdots
\end{equation*}
The mixed Hodge structures on $H^k(X_i), H^k(X_{i+1}), H^k(C_i), H^k(E_i)$ are pure of weight $k$. Therefore, we have an exact sequence of pure Hodge structures:
\begin{equation*}
0\longrightarrow H^k(X_i)\longrightarrow H^k(X_{i+1})\oplus H^k(C_i)\longrightarrow H^k(E_i)\longrightarrow 0
\end{equation*}
To show $H^{k,0}(X_i)\cong H^{k,0}(X_{i+1})$, it suffices to show that $H^{k,0}(C_i)\cong H^{k,0}(E_i)$. This is true because $E_i$ is a $\PP^1$-bundle over $C_i$.
	\end{proof}

We give the following criterion for the completeness of $\X_T\to\sU_T$:
\begin{prop}
\label{proposition: complete}
Suppose $T=(L_1, \cdots, L_m)$ is partition of ball type in Theorem \ref{theorem: main}. Then the family $\X_T\to\sU_T$ of Calabi--Yau manifolds is complete if and only if any transversal intersection $D_i\cap D_j$ (for $D_i\in |L_i|, D_j\in |L_j|$) satisfies $H^{n-2,0}(D_i\cap D_j)=0$, namely, $D_i\cap D_j$ admits no nonzero holomorphic $(n-2)$-forms.
	\end{prop}

\begin{proof}
The family $\X_T\to\sU_T$ of Calabi--Yau manifolds is complete around $D\in\sU_T$ if 
		\begin{equation*}
  T_D{\sU_T}\to H^1(X, T_X) \cong \Hom(H^{n,0}(X), H^{n-1,1}(X))
		\end{equation*}
is surjective. From Theorem \ref{theorem: local torelli}, the map $$T_D{\sU_T}\to H^1(Y,\calT_Y)\cong \Hom(H^{n,0}(Y), H^{n-1,1}_\chi(Y))$$ is surjective. If $T$ is a ball-type partition in Theorem \ref{theorem: main}, then $H^{n-1,1}_\chi(Y)\cong H^{n-1,1}(Y)$. By Proposition \ref{proposition: crepant is CY}, the crepant resolution $X\to Y$ can be decomposed as $X=X_N\to X_{N-1}\to\cdots\to X_1=X$, such that each step $f_i\colon X_{i+1}\to X_i$ is the blowup along a center $C_i\subset X_i$ of codimension two. So $H^{n-1,1}(X_{i+1})\cong H^{n-1,1}(X_i)\oplus H^{n-2,0}(C_i)[1]$. Thus $H^{n-1,1}(Y)\cong H^{n-1,1}(X)$ if and only if all centers $C_i$ satisfy $H^{n-2,0}(C_i)=0$, namely $C_i$ do not have global nonzero holomorphic $(n-2)$-forms. We claim that this is true if and only if all intersections $D_i\cap D_j$ do not admit nonzero global holomorphic $(n-2)$-forms.

On one hand, each $D_i\cap D_j$ is birational to certain $C_i$. Notice that $H^{n-2,0}$ is a birational invariant. So if $C_i$ do not admit nonzero global holomophic $(n-2)$-forms, then neither do $D_i\cap D_j$.

On the other hand, each $C_i$ is either birational to certain $D_i\cap D_j$, or birational to the intersection of the exceptional divisor of certain $f_j\colon X_{j+1}\to X_j$ with the strict transformation of an irreducible divisor in $X_j$, and in our case such an intersection must be a $\PP^1$-bundle over a smooth projective variety of dimension $n-3$, hence does not admit nonzero global holomorphic $(n-2)$-forms. Therefore, if any $D_i\cap D_j$ do not admit nonzero global holomorphic $(n-2)$-forms, then neither do $C_i$.
\end{proof}
	
\begin{cor}
Let (Z, d, L, T) be ball-type cases in Theorem \ref{theorem: main}. If $Z=(\PP^1)^3$, those $T$ such that $\X_T\to \sU_T$ is complete are listed in Table \ref{table: CY Deligne--Mostow}. If $Z=(\PP^1)^n$ with $n=2$ or $\ge 4$, then any ball-type partition $T$ is complete.
\end{cor}
\begin{proof}
This is direct from Proposition \ref{proposition: complete}. When $Z=(\PP^1)^3$, the type is complete if and only if any component of $D_i\cap D_j$ is rational. The last column of Table \ref{table: CY Deligne--Mostow} is by straightforward calculation. The case $Z=\PP^2$ is clear.

For $Z=(\PP^1)^n$ with $n\ge 4$. Let $T$ be a partition of $\calO(3,\cdots, 3)$ of ball type. Then the sum of any two divisors has at least $n-3\ge 1$ components equal to zero. Thus the corresponding intersection $D_i\cap D_j$ is a $\PP^1$-bundle over a smooth projective variety (may be disconnected). Therefore $H^{n-2,0}(D_i\cap D_j)=0$. So $T$ is always complete.
\end{proof}
	
\section{Period Map and Monodromy Group}
\label{section: period map}
In this section we prove the arithmeticity of the monodromy groups in the ball type cases in Theorem \ref{theorem: main}. We first prove the following general result.

\begin{prop}
\label{proposition: arithmeticity general}
Suppose  $\DD$ is a Hermitian symmetric domain of noncompact type and $\Gamma\subset \Aut(\DD)$ an arithmetic subgroup. Suppose $\calM$ to be a smooth quasi-projective complex variety with universal cover $\calM^u$. Suppose $\rho\colon \pi_1(\calM)\to \Gamma$ is a group homomorphism. Suppose $p\colon \calM^u\to \DD$ is a $\rho$-equivariant holomorphic map. Furthermore, assume that $p$ is locally biholomorphic at generic points in $\calM^u$. Then $[\Gamma:\rho(\pi_1(\calM))]<\infty$. In particular, $\rho(\pi_1(\calM))$ is arithmetic.
\end{prop}

\begin{proof}
The arithmetic group $\Gamma$ admits a torsion-free normal subgroup, say $\Gamma_1$. Let $H=\rho^{-1}(\Gamma_1)\triangleleft \pi_1(M)$. There exists a covering $\widetilde{\calM}\to\calM$ such that the image of $\pi_1(\widetilde{\calM})\hookrightarrow \pi_1(\calM)$ is equal to $H$. The finiteness of $[\Gamma:\rho(\pi_1(\calM))]$ and $[\Gamma_1:\rho(H)]$ are equivalent. Therefore, we may assume (without loss of generality) that $\Gamma$ is torsion-free.

Denote by $\Prd$ the analytic map $\calM\to \Gamma\bs \DD$ induced by $p\colon \calM^u\to \DD$ (after taking quotient of $\rho\colon \pi_1(\calM)\to \Gamma$). Borel's extension theorem (\cite{borel1972some})  enables us to extend $\Prd$ to an algebraic morphism
\begin{equation*}
\Prd\colon \overline{\calM}\to \overline{\Gamma \bs\DD}^{bb}
\end{equation*}
between two complex projective varieties. Here $\overline{\calM}$ is a compactification of $\calM$ with $\overline{\calM}-\calM$ a normal crossing boundary divisor, and $\overline{\Gamma \bs\DD}^{bb}$ is the Baily-Borel compactification of $\Gamma\bs\DD$.

The extended map $\Prd\colon\overline{\calM}\to \overline{\Gamma \bs\DD}^{bb}$ is locally biholomorphic at generic points in $\overline{\calM}$. Thus there exists a smooth open affine subvariety $V \subset  \Gamma \bs\DD$ such that $\Prd^{-1}V\subset \calM$ is also affine and $\Prd\colon \Prd^{-1}V\to V$ is a finite morphism.  We may also ask (after shrinking $V$ if necessary) $\Prd$ to be locally biholomorphic at any points in $\Prd^{-1}V$. Now $\Prd\colon \Prd^{-1}V\to V$ is proper (namely, the preimage of a compact set is compact) and is locally bihomeomorphism everywhere, hence (by standard topology theory) a finite covering map (in the canonical topological sense). This implies that the group morphism $\pi_1(\Prd^{-1}V)\to \pi_1(V)$ is injective and has finite index.

We have the following diagram of group morphisms:
\begin{equation*}
\begin{tikzcd}
\pi_1(\Prd^{-1}V)\arrow{r}\arrow[d, twoheadrightarrow] & \pi_1(V) \arrow[d, twoheadrightarrow] \\
\pi_1(M)\arrow{r}{\rho} & \pi_1(\Gamma\bs \DD)=\Gamma
\end{tikzcd}
\end{equation*}
where the two vertical maps are surjective (see e.g. \cite[Theorem 2.1]{arapura2016fundamental}). This clearly implies that $[\Gamma:\rho(\pi)]<\infty$.
\end{proof}

We make the following conjecture which generalize \cite[Proposition 5.3]{mostow1988discontinuous}.
\begin{conj}
Suppose $\calM$ is a smooth quasi-projective complex variety and $\calM^u$ its universal covering. Suppose $\DD$ to be a Hermitian symmetric domain of noncompact type. Suppose $p\colon\calM^u\to \DD$ to be a holomorphic map with an equivariant action of $\pi_1(\calM)$. Let $\rho\colon \pi_1(M)\to \Aut(\DD)$. Assume that $p$ is locally biholomorphic at generic points of $\calM^u$. Assume moreover that $\rho(\pi_1(M))$ is discrete in $\Aut(\DD)$. Then $\rho(\pi_1(M))$ is a lattice.
\end{conj}

Now suppose $(Z,d,L,T)$ satisfies Condition \ref{condition}. Suppose $Z=\PP^{n_1}\times\cdots\times \PP^{n_t}$. By Proposition \ref{proposition: stability}, a generic point of $\PP_T=\prod\limits_{i=1}^m |L_i|$ is stable under $G=\SL(n_1+1)\times\cdots\times \SL(n_t+1)$. Let $\sU^\circ_T\subset\PP_T$ be the subset of stable and normal crossing divisors of type $T$. Define 
\begin{equation*}
\calM_T\coloneqq \sU^\circ_T//G
\end{equation*}
to be the GIT quotient.
We have a monodromy representation 
\begin{equation*}
\pi_1(\sU^\circ_T)\to \PU(H^n(Y, \ZZ[\xi_d])_\chi, h)\xhookrightarrow{} \PU(H^n(Y, \CC)_\chi, h)
\end{equation*}
with the image denoted by $\Gamma_T$. So we obtain a period map $p\colon \sU^\circ_T \to \Gamma_T \bs \BB^k$.  Then $p$ descends to a period map $\Prd\colon\calM_T\to \Gamma_T\bs\BB^k$. Using Proposition \ref{proposition: arithmeticity general}, we can deduce the arithmeticity of $\Gamma_T$ if $d=3,4,6$.

\begin{prop}
\label{proposition: monodromy}
Assume $d=3,4,6$, the monodromy group $\Gamma_T$ is an arithmetic subgroup of $\PU(H^n(Y, \CC)_\chi, h)\cong \PU(1,k)$.
\end{prop}
\begin{proof}
Since $d=3,4,6$, we have $[\QQ[\xi_d]\colon \QQ]=2$. Thus the group $\Gamma=\PU(H^n(Y, \ZZ[\xi_d])_\chi, h)$ is an arithmetic subgroup of $\PU(H^n(Y, \CC)_\chi, h)$. Then we have analytic maps $\calM_T\to \Gamma_T\bs \BB^k\to \Gamma\bs \BB^k$. Theorem \ref{theorem: local torelli} implies that the tangent map of $\calM_T\to \Gamma_T \bs \BB^k$ at a generic point in $\calM_T$ is an isomorphism. By Proposition \ref{proposition: arithmeticity general}, $\Gamma_T$ has finite index in $\Gamma$ and hence also an arithmetic subgroup of $\PU(1,k)$.
\end{proof}

The Global Torelli for Calabi--Yau varieties in ball-type cases is not known in general. One step would be to study the degree of maps $\calM\to \Gamma_T\bs \BB$. We make the following conjecture: 
\begin{conj}
Let $Z$ be a product of projective spaces, and $(Z,d,L,T)$ a tuple satisfying Condition \ref{condition}. Suppose $\BB_T$ is of ball type, then any elements in $\sU_T$ are stable, and the period map $\Prd\colon \calM_T \to \Gamma_T\bs \BB$ is an open embedding.
\end{conj}

Special cases were studied in \cite{sheng2019global} via relating the monodromy groups to Deligne--Mostow theory.

\section{Relations to Deligne--Mostow Theory: Dimension $n=2$}
	\label{section: relation to DM}
	In \cite{kondo2000complex}, Kond\=o studied the case of ball quotient given by $(Z,d,L)=((\PP^1)^2, 3, \calO(1)^{\boxtimes 2})$. Since a generic curve of genus four can be realized as a curve in $(\PP^1)^2$ of degree $(3, 3)$, the cyclic cover construction equips the moduli of curves of genus four with a ball-quotient structure. From the analysis of the transcendental lattice of the corresponding $K3$ surfaces, Kond\=o proved that this ball quotient is commensurable with the ball quotient given by Deligne--Mostow \cite{deligne1986monodromy} with type $$\mu=({1\over 6},{1\over 6},{1\over 6},{1\over 6},{1\over 6},{1\over 6},{1\over 6},{1\over 6},{1\over 6},{1\over 6},{1\over 6},{1\over 6}).$$

In this section we will study commensurable relations between arithmetic lattices in Deligne--Mostow theory arising from degenerations of Kond\=o's construction. As a byproduct, we give another proof of Kond\={o}'s result. The methods will be extended in \S\ref{section: 3 DM} to study commensurable relations among arithmetic subgroups arising from the case $Z=(\PP^1)^3$.

\subsection{Monodromy of Elliptic Fibrations}
\label{section: monodromy of Kondo}
In this section, we calculate the monodromy of elliptic fibrations arising from degenerations of Kond\=o's example. We first consider the case of type $(3,1)+(0,2)$, where all types of singular fibers in other degenerations arise. Let $C_1\in |(3, 1)|$ and $C_2\in |(0,2)|$ be two generic curves in $(\PP^1)^2=P_1\times P_2$. We denote the projection $P_1\times P_2\to P_i$ by $p_i$. The induced projections from $C_j$ and $S$ to $P_i$ are also denoted by $p_i$. We demonstrate the curves $C_1$ and $C_2$ (in a generic way) in Figure \ref{figure: configuration}.
\begin{figure}[htp]
    \centering
    \includegraphics[width=4.5cm]{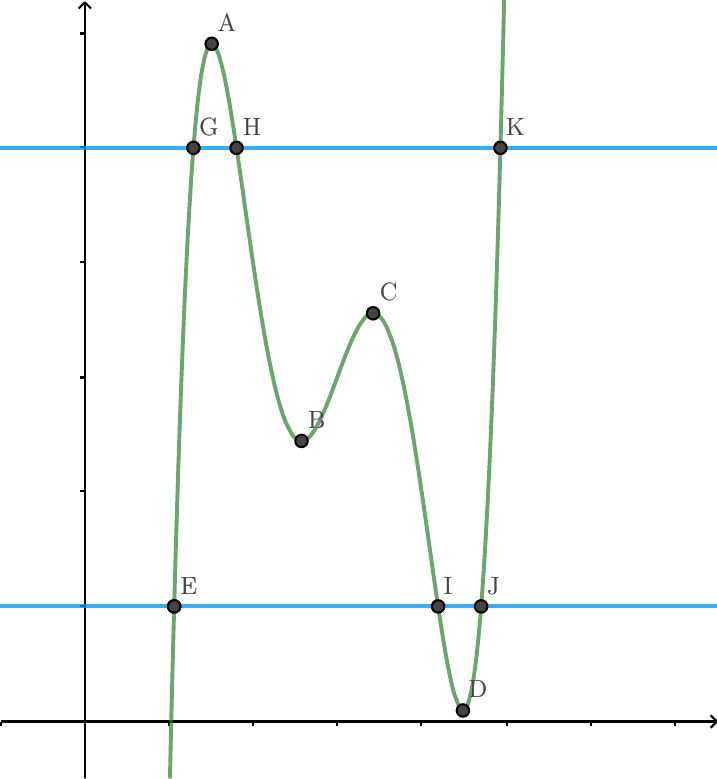}
    \caption{$(3,1)+(0,2)$}
    \label{figure: configuration}
\end{figure}
The intersection of $C_1$ and $C_2$ are $G, H, K, E, I, J$, and their projections by $p_1$ form $6$ distinct points $a_1, a_2, \cdots, a_6$ on $P_1$. On the other hand, the projection $p_2\colon C_1\to P_2$ has $4$ ramification points $A, B, C, D$ and their projections by $p_2$ form $4$ points $b_1, \cdots, b_4$ on $P_2$. The projection $p_2(C_2)$ contains two points $c_1, c_2$. Points $b_1, \cdots, b_4, c_1, c_2$ are distinct.

The triple cover of $P_1\times P_2$ is a $K3$ surface $S$ with $6$ du Val singularities of type $A_2$ at the preimage of $G, H, K, E, I, J$. Resolution of these singularities yields six chains of exceptional $\PP^1$'s with length $2$. 

\begin{prop}
\label{proposition: isotrivial elliptic fibrations and Kodaira types}
The smooth resolution of $S$ is a smooth $K3$ surface $\widetilde{S}$ with two isotrivial elliptic fibrations $p_1\colon \tilde{S}\to P_1$ and $p_2\colon \widetilde{S}\to P_2$. 
\begin{enumerate}
    \item The elliptic fibration $p_1\colon \widetilde{S}\to P_1$ has $6$ singular fibers over $a_1, \cdots, a_6$ of Kodaira type $\IV$.
    \item The elliptic fibration $p_2\colon \widetilde{S}\to P_2$ has $6$ singular fibers over $b_1, \cdots, b_4$, $c_1$, $c_2$. The four singular fibers over $b_1, \cdots, b_4$ are of type II, and the two singular fibers over $c_1, c_2$ are of type $\IV^*$.
\end{enumerate} 
    
\end{prop}
\begin{proof}
The fiber of $p_1\colon S\to P_1$ over $a_i\in P_1$ is a rational curve with a cusp. The fiber of $p_1\colon \widetilde{S}\to P_1$ over $a_i$ is a chain of $\PP^1$'s with dual graph $\widetilde{A}_2$ 
\begin{equation*}
\dynkin[extended, edge length=0.8cm]A{2}
\end{equation*}
while the three $\PP^1$'s intersecting at one point.

The fiber of $p_2\colon S\to P_2$ over $b_i\in P_2$ is a rational curve with a cusp. The resolution of $S$ does not affect these four fibers. Namely, the fiber of $p_2\colon \widetilde{S}\to P_2$ over $b_i$ is still the rational curve with a cusp.

The fiber of $p_2\colon S\to P_2$ over $c_i\in P_2$ is a $\PP^1$. The fiber of $p_2\colon \widetilde{S}\to P_2$ over $c_i$ is a chain of $\PP^1$'s with dual graph $\widetilde{E}_6$
\begin{equation*}
\dynkin [extended, edge length=0.5cm]E6.
\end{equation*}

See \cite[Table 3, \S V.10]{barth2015compact} for the type of these singular fibers.  
\end{proof}

\begin{cor}
\label{corollary: monodromy around singular fibers}
   \begin{enumerate}
       \item The $\chi$-eigensheaf $\LL_1=R^1{p_1}_*(\QQ[\omega])_\chi|_{P_1-\{a_1,\cdots, a_6\}}$ is a complex rank-one local system with monodromy of multiplication with $\omega=\exp({2\pi\sqrt{-1}\over 3})$ around each $a_i$. 
       \item The $\chi$-eigensheaf $\LL_2=R^1{p_2}_*(\QQ[\omega])_\chi|_{P_2-\{b_1,\cdots, b_4, c_1, c_2\}}$ is a complex rank-one local system with monodromy of multiplication with $-\omega^2=\exp({\pi\sqrt{-1}\over 3})$ around each $b_j$, and multiplication with $\omega^2=\exp({4\pi\sqrt{-1}\over 3})$ around each $c_k$.
   \end{enumerate} 
\end{cor}

\begin{proof}
This directly follows from \cite[\S V.10]{barth2015compact}, especially Table 6. We explain this in a concrete way. 

Let $E=\CC/(\ZZ+\ZZ\omega)$ be the elliptic curve with j-invariant zero. Let $D$ be the unit disk in $\CC$. Consider the action of $\mu_3$ on $E\times D$ with $\omega(z,t)=(\omega z, \omega t)$. Let $S$ be the minimal resolution of the quotient of $E\times D$ by $\mu_3$. Then the fibration 
\begin{equation*}
S\to D/\mu_3
\end{equation*}
has fiber of type $\IV$ over $0$. The monodromy on $E$ induced by a counter-clockwise loop in $D/\mu_3$ is given by $g\colon z\mapsto \omega^{-1} z$. The space $H^1(E,\CC)_\chi=H^{1,0}(E)$ is a complex line generated by $dz$, and the induced monodromy on $H^1(E,\CC)_\chi$ is given by multiplication with $\omega$.

If we change the action to be $\omega(z,t)=(\omega^2 z, \omega t)$, then the central fiber of $S\to D/\mu_3$ is of type $\IV^*$. In this case, the monodromy on $H^1(E,\CC)_\chi$ is given by multiplication with $\omega^2$.

If we change the action to be $\omega(z,t)=(-\omega^2 z, -\omega^2 t)$, then the central fiber of $S\to D/\mu_3$ is of type II. In this case, the monodromy on $H^1(E,\CC)_\chi$ is given by multiplication with $-\omega^2$.
\end{proof}

From Corollary \ref{corollary: monodromy around singular fibers}, the local system $\LL_1$ on $P_1-\{a_1,\cdots, a_6\}$ corresponds to the Deligne--Mostow data 
\begin{equation*}
\alpha_1=({1\over 3}, {1\over 3},{1\over 3},{1\over 3},{1\over 3},{1\over 3}).
\end{equation*}
and the local system $\LL_2$ over $P_2-\{b_1, \cdots, b_4, c_1, c_2\}$ corresponds to the Deligne--Mostow data 
\begin{equation*}
\alpha_2=({2\over 3}, {2\over 3}, {1\over 6}, {1\over 6}, {1\over 6}, {1\over 6}).
\end{equation*}

For other kinds of degenerations, similar argument gives rise to the following table:

\begin{longtable}
{|p{.30\textwidth} | p{.20\textwidth} | p{.20\textwidth} |}
\hline
Partition Type & DM Data $6\alpha_1$ & DM $6\alpha_2$ \\\hline
$(3,3)$ & $(1^{12})$ & $(1^{12})$ \\\hline

$(3,2)+(0,1)$ & $(1^6,2^3)$ & $(1^8,4)$ \\\hline
$(3,1)+(0,2)$ & $(2^6)$ & $(1^4,4^2)$ \\\hline

$(2,2)+(1,1)$ & $(1^4,2^4)$ & $(1^4,2^4)$ \\\hline

$(2,2)+(1,0)+(0,1)$ & $(1^4,2^2,4)$ & $(1^4,2^2,4)$ \\\hline

$(2,1)+(1,2)$ & $(1^2,2^5)$ & $(1^2,2^5)$ \\\hline

$(2,1)+(1,1)+(0,1)$ & $(2^6)$ & $(1^2,2^3,4)$ \\\hline

$(2,1)+(1,0)+(0,2)$ & $(2^4,4)$ & $(1^2,2,4^2)$ \\\hline

$(2,0)+(1,1)+(0,2)$ & $(2^2,4^2)$ & $(2^2,4^2)$ \\\hline

$(1,1)+(1,1)+(1,1)$ & $(2^6)$ & $(2^6)$ \\\hline

$(1,1)+(1,1)+(1,0)+(0,1)$ & $(2^4,4)$ & $(2^4,4)$ \\\hline

$(3,0)+(0,3)$ & $(4^3)$ & $(4^3)$ \\\hline
\caption{Deligne--Mostow Data in Two-Dimensional Cases}
\label{table: n=2}
\end{longtable}

\subsection{Non-SNC Cases}
In this section, we calculate the monodromy of fibration when the divisors are not simple normal crossing. We call these non-SNC cases. We start with the example from the case of type $(3,1)+(0,2)$ and use the same notation and picture in \S\ref{section: monodromy of Kondo}. The curve $C_2$ is the union of two lines. Consider the moduli of $C_1$ and $C_2$ such that one of the lines in $C_2$ is tangent to curve $C_1$. In terms of points $a_1,a_2,\cdots, a_6$, this corresponds to colliding of two of them. In the other projection, a point $b_i$ collides with certain $c_j$. A generic surface $S$ with such type of $C_1$ and $C_2$ has two elliptic fibrations over $P_1$ and $P_2$, both with five singular fibers. The monodromy data of the two elliptic fibrations around those five points are $\alpha_1=({2\over 3}, {1\over 3}, {1\over 3}, {1\over 3}, {1\over 3})$ and $\alpha_2=({5\over 6},{2\over 3}, {1\over 6}, {1\over 6}, {1\over 6})$.

Similarly, consider the type $(3,2)+(0,1)$ with $C_1, C_2$ tangent at a point. The corresponding monodromy data are $\alpha_1=({2\over 3},{1\over 3},{1\over 6}, {1\over 6}, {1\over 6},{1\over 6}, {1\over 6}, {1\over 6})$ and $\alpha_2=({5\over 6},{1\over 6},{1\over 6}, {1\over 6}, {1\over 6},{1\over 6}, {1\over 6}, {1\over 6})$.


We collect all pairs of monodromy data from non-SNC cases in the following table.
\begin{longtable}
{|p{.50\textwidth} | p{.20\textwidth} | p{.20\textwidth} |}
\hline
Degeneration Type & DM Data $6\alpha_1$ & DM $6\alpha_2$ \\\hline

$C_1:(3,2)$, $C_2:(0,1)$. $C_1$ and $C_2$ are tangent. & $(1^6,2,4)$ & $(1^7,5)$ \\\hline
$C_1: (3,2), C_2: (0,1)$, $C_1$ has one node & $(1^4,2^4)$ & $(1^6,2,4)$ \\\hline
$C_1:(3,2),C_2:(0,1)$. $C_1$ has one node. $C_1$ and $C_2$ are tangent. & $(1^4,2^2,4)$ & $(1^5,2,5)$ \\\hline
$C_1: (3,2), C_2: (0,1)$. $C_1$ has two nodes. & $(1^2,2^5)$ & $(1^4,2^2,4)$ \\\hline
$C_1: (3,2), C_2: (0,1)$. $C_1$ has two nodes. $C_1$ and $C_2$ are tangent. & $(1^2,2^3,4)$ & $(1^3,2^2,5)$ \\\hline
$C_1:(3,2), C_2:(0,1)$. $C_1$ has three nodes. & $(2^6)$ & $(1^2,2^3,4)$ \\\hline
$C_1:(3,2), C_2:(0,1)$. $C_1$ has three nodes. $C_1$ and $C_2$ are tangent. & $(2^4,4)$ & $(1,2^3,5)$ \\\hline
$C_1:(3,1), C_2:(0,1), C_3:(0,1)$. $C_1$ and $C_2$ are tangent. & $(2^4,4)$ & $(1^3,4,5)$ \\\hline
$C_1:(3,1), C_2:(0,1), C_3:(0,1)$. $C_1$ and $C_2$ are tangent. $C_1$ and $C_3$ are tangent. & $(2^2,4^2)$ & $(1^2,5^2)$ \\\hline
$C_1:(2,2), C_2:(1,0), C_3:(0,1)$. $C_1$ and $C_2$ are tangent. & $(1^3,2^2,5)$ & $(1^4,4^2)$ \\\hline
$C_1:(2,2), C_2:(1,0), C_3:(0,1)$. $C_1$ has one node. $C_1$ and $C_2$ are tangent. & $(1,2^3,5)$ & $(1^2,2,4^2)$ \\\hline
$C_1:(2,1), C_2:(1,1), C_3:(0,1)$. $C_1$ and $C_3$ are tangent & $(2^4,4)$ & $(1,2^3,5)$ \\\hline
$C_1:(2,1), C_2:(1,1), C_3:(0,1)$. $C_1$ and $C_2$ are tangent & $(2^4,4)$ & $(1^2,2,4^2)$ \\\hline
$C_1:(2,1), C_2:(1,0),C_3:(0,1),C_4:(0,1)$. $C_1$ and $C_3$ are tangent & $(2^2,4^2)$ & $(1,2,4,5)$ \\\hline
\caption{Deligne--Mostow Data in Two-Dimensional and Non-SNC Cases}
\label{table: nonnsc}
\end{longtable}

\subsection{Identification of $\QQ[\omega]$-Hodge Structures of Ball Type}

Suppose that we are given an elliptic fibration $\pi\colon S\to \PP^1$ such that $S$ is a $K3$ surface, each fiber has zero j-invariant, and $\LL\coloneqq R^1\pi_*\QQ[\omega]$ corresponds to certain Deligne--Mostow data $\alpha=(\alpha_1, \cdots, \alpha_k)$. There is a non-symplectic action of $\mu_3$ on $S$. We have the $\chi$-characteristic subspace $H^2(S, \QQ[\omega])_\chi$. There is a Hermitian pairing with signature $(1,k-3)$ on $H^2(S, \QQ[\omega])_\chi$ and Hodge filtration given by $F^1=H^0(S, K_S)\subset H^2(S, \CC)_\chi=F^0$. 

On the other hand, denote by $e=3 \text{ or }6$ the common denominator of $\alpha_i$ and consider $C$ the Deligne--Mostow cyclic covering of $\PP^1$ given by the affine equation
\begin{equation*}
y^e=\prod_i (x-x_i)^{e\alpha_i}.
\end{equation*}
By Deligne--Mostow theory, for the tautological character $\psi\colon \mu_e\hookrightarrow\CC^{\times}$, the characteristic subspace $H^1(C, \QQ[\omega])_{\overline{\psi}}$ has a Hermitian pairing of signature $(1, k-3)$ with Hodge filtration given by $F^1=H^0(C, K_C)_{\overline{\psi}}\subset H^1(C, \CC)_{\overline{\psi}}=F^0$. 

Let $A=\{a_1,\cdots, a_k\}$ be the branch locus of $\pi\colon S\to \PP^1$. Define $\LL=(R^1\pi_*\QQ[\omega])_\chi$. By \cite[\S 2.23]{deligne1986monodromy}, we have an isomorphism 
\begin{equation}
\label{isomorphism deligne mostow}
H^1(\PP^1-A, \LL)\cong H^1(C, \QQ[\omega])_{\overline{\psi}} 
\end{equation}
of $\QQ[\omega]$-Hodge structures of ball type. 
\begin{prop}
\label{proposition: kondo to Deligne--Mostow}
There is a natural isomorphism 
\begin{equation*}
	H^2(S, \QQ[\omega])_\chi\cong H^1(\PP^1-A, \LL)
\end{equation*}	
between $\QQ[\omega]$-Hodge structures of ball type.
\end{prop}
\begin{proof}
Let $U=p^{-1}(\PP^1-A)$. Denote $j\colon \PP^1-A \to \PP^1$ the inclusion map. Then the natural morphism $g\colon \LL\to j_* j^* \LL$ is an isomorphism over $\PP^1-A$. So the kernel of $g$ is supported on $A$. Since the monodromy of $j^* \LL$ around each $a_i\in A$ is nontrivial, the stalks of $j_* j^* \LL$ at $a_i$ are zero. So the morphism $g$ is surjective over $\PP^1$. The long exact sequence associated with 
\begin{equation*}
0\to \mathrm{Ker}(g)\to \LL\to j_* j^*\LL\to 0 
\end{equation*}
implies that $j^*\colon H^1(\PP^1, \LL) \to H^1(\PP^1-A, \LL)$ is an isomorphism. 

Consider the Leray spectral sequence
\begin{equation*}
E^{p,q}=H^p(\PP^1, R^q \pi_*\QQ[\omega]) \implies H^{p+q}(S, \QQ[\omega]).
\end{equation*}
Since $\mu_3$ operates on the fibers, we have
\begin{equation*}
E^{p,q}_\chi=H^p(\PP^1, (R^q\pi_*\QQ[\omega])_\chi)\implies H^{p+q}(S, \QQ[\omega])_\chi.
\end{equation*}
The $\chi$-eigensubsheaves $(R^0\pi_*\QQ[\omega])_\chi$ and $(R^2\pi_*\QQ[\omega])_\chi$ are apparently zero. Taking $p+q=2$, we obtain $H^2(S, \QQ[\omega])_\chi\cong H^1(\PP^1, \LL)$. Similarly, we have $H^2(U, \QQ[\omega])_\chi\cong H^1(\PP^1-A, \LL)$ and $H_c^2(U, \QQ[\omega])_\chi\cong H_c^1(\PP^1-A, \LL)$. So the four items in the commutative diagram
\begin{equation}
	\label{diagram: cohomology of K3 to local system}
	\begin{tikzcd}
	H^2(S, \QQ[\omega])_\chi \arrow{r}\arrow{d} & H^2(U, \QQ[\omega])_\chi \arrow{d}\\
	H^1(\PP^1, \LL) \arrow{r} & H^1(\PP^1-A, \LL)
	\end{tikzcd}
\end{equation}
are isomorphic to each other. The same argument gives isomorphism 
\begin{equation*}
H^1_c(\PP^1-A, \LL^\vee)\cong H^2_c(U, \QQ[\omega])_{\overline{\chi}}\cong H^2(S, \QQ[\omega])_{\overline{\chi}}. 
\end{equation*}
Moreover the Poincar\'e pairings
\begin{equation}
\label{equation: poincare S}
  H^2(S, \QQ[\omega])_\chi\times H^2(S, \QQ[\omega])_{\overline{\chi}}\to \QQ[\omega]
\end{equation}
and
\begin{equation}
\label{equation: poincare U}
  H^2(U, \QQ[\omega])_
  \chi\times H^2_c(U, \QQ[\omega])_{\overline{\chi}}\to \QQ[\omega]
\end{equation}
are the same under the induced isomorphisms.

There is a nondegenerate bilinear form 
\begin{equation}
\label{equation: poincare P1-A}
H^1(\PP^1-A, \LL)\times H^1_c(\PP^1-A, \LL^\vee)\to \QQ[\omega].
\end{equation}
compatible with the Poincar\'e pairings \eqref{equation: poincare S} and \eqref{equation: poincare U}. The quasi-isomorphism $j_*\LL^\vee \cong j_!\LL^\vee$ (see \cite[Proposition 2.6.1]{deligne1986monodromy}) induces isomorphisms 
\begin{equation*}
H^1_c(\PP^1-A, \LL^\vee)\cong H^1(\PP^1-A, \overline{\LL})\cong \overline{H^1(\PP^1-A, {\LL})}.
\end{equation*}
Therefore, the Poincar\'e pairing \eqref{equation: poincare P1-A} gives rise to a Hermitian form on $H^1(\PP^1-A, \LL)$. This is compatible with the Hermitian form on $H^2(S, \QQ[\omega])_\chi$ given by \eqref{equation: poincare S} and identification $H^2(S, \QQ[\omega])_{\overline{\chi}}= \overline{H^2(S, \QQ[\omega])_\chi}$. The Hodge filtration on $H^2(S, \QQ[\omega])_{\overline{\chi}}$ is
\begin{equation*}
F^1=H^0(S, K_S)\cong H^0(\PP^1, \pi_*K_S)\cong H^0(\PP^1, \Omega^1(\sum_{i=1}^k \alpha_i a_i)(\LL)).
\end{equation*}
See \cite[\S2.11]{deligne1986monodromy} for more details about the line bundle $\Omega^1(\sum_{i=1}^k \alpha_i a_i)(\LL)$. This implies that the Hodge filtrations are compatible with the isomorphism
\begin{equation*}
H^2(S, \QQ[\omega])_\chi\cong H^1(\PP^1-A, \LL).
\end{equation*}
So this is an isomorphism between $\QQ[\omega]$-Hodge structures of ball-type.
\end{proof}

\begin{cor}
There is a natural isomorphism
\begin{equation*}
H^2(S, \QQ[\omega])_\chi\cong H^1(C, \QQ[\omega])_{\overline{\psi}},
\end{equation*}	
between $\QQ[\omega]$-Hodge structures of ball type.
\end{cor}
\begin{proof}
This follows directly from \eqref{isomorphism deligne mostow} and Proposition \ref{proposition: kondo to Deligne--Mostow}.
\end{proof}

\begin{cor}
\label{corollary: comm 2}
For each case $T$ in Table \ref{table: n=2} and \ref{table: nonnsc}, the arthmetic group $\PU(H^2(S,\ZZ[\omega])_\chi,h)$ and the Deligne--Mostow groups corresponding to $\alpha_1, \alpha_2$ are commensurable.
\end{cor}


\section{Relations to Deligne--Mostow Theory: Dimension $n>2$}
\label{section: 3 DM}
In this section we study the ball quotients arising from $(Z, d, L)=((\PP^1)^3, 3, \calO(1)^{\boxtimes 3})$ towards their relation to Deligne--Mostow theory. There are $40$ cases of ball type, and we will show $31$ of them are related to Deligne--Mostow ball quotients via isotrivial fibration of $K3$ surfaces (see Table \ref{table: CY Deligne--Mostow}). The case of $T=((3,3,0), (0,0,3))$ has been well-studied by Voisin-Borcea-Rohde \cite{voisin1993miroirs, borcea1997k3, rohde2009cyclic}. In fact any example of $T$ with a summand $(0,0,3)$ can be related to Kond\=o's $K3$ surfaces studied in \S\ref{section: monodromy of Kondo} via van Geemen's half-twist of Hodge structures \cite{van2001half, van2002half}, see also \S\ref{subsection: half-twist}. On the other hand, there are examples not directly from the half-twist construction. We study the case $T=((3,2,0),(0,1,3))$ in detail and other examples follows similarly. The correspondence with Deligne--Mostow ball quotients are collected in Table \ref{table: CY Deligne--Mostow}.

We write $Z=P_1\times P_2 \times P_3$, where $P_i$ are projective lines. Consider type $T=(3,2,0)+(0,1,3)$ and let $D=D_1+D_2$ be a simple normal crossing divisor on $Z$ of type $T$. The Calabi--Yau orbifold $Y$ is the triple cover of $Z$ branched along $D$. Let $p_i$ be the projection of $Y$ or $Z$ to $P_i$.

Let $D_1=C_1\times P_3$ with $C_1$ a $(3,2)$-curve in $P_1\times P_2$, and let $D_2=P_1\times C_2$ with $C_2$ a $(1,3)$-curve in $P_2\times P_3$. Let $pr_1\colon C_1\to P_2$ and $pr_2\colon C_2\to P_2$ be natural projections. We have $g(C_1)=2$ and $g(C_2)=0$. Both $pr_1$ and $pr_2$ are triple covering maps. By Hurwitz formula, $pr_1$ has $8$ branch points $x_1, \cdots, x_8$ and $pr_2$ has $4$ branch points $x_9, \cdots, x_{12}$. For a generic choice of $D_1$ and $D_2$, the $12$ branch points are different. The fibration $p_2\colon Y\to P_2$ is an isotrivial family of ADE $K3$ surfaces outside $B=\{x_1, \cdots, x_{12}\}$. The generic fiber is a cyclic cover of $P_1\times P_3$ branching along divisors of type $(3,0)+(0,3)$. Consider the $\chi$-eigensheaf 
\begin{equation*}
\LL=(R^2 {p_2}_* \QQ[\omega])_\chi|_{P_2-\{x_1,\cdots, x_{12}\}}.
\end{equation*}

\begin{prop}
    The sheaf $\LL$ is a rank-one $\QQ[\omega]$-local system with monodromy corresponding to Deligne--Mostow data
    $\alpha=({1\over 6})^{12}.$
\end{prop}
\begin{proof}
Let $E_i\to P_i$ be a triple cover of $P_i$ branching at three distinct points. The cyclic group operation $\mu_3$ on $E_i$ is denoted by $a\cdot x$. The product group $\mu_3\times \mu_3=\{(a, b)|a, b\in \mu_3\}$ operates on $E_1\times E_3$ by $(a, b)\cdot (x, y)=(a\cdot x, b\cdot y)$. This induces an operation of $\mu_3$ on $E_1\times E_3$ defined by 
\begin{equation}
\label{equation: skewed-diagonal action}
    a\cdot (x, y)=(a\cdot x,a^{-1}\cdot y).
\end{equation} 
The quotient space is denoted by $\mu_3\bs E_1\times E_3$ and it has a $\mu_3$-operation defined by $a\cdot [(x, y)]=[(a\cdot x, y)]$.

The fibers of $p_2$ outside $B=\{x_1, \cdots, x_{12}\}$ are $K3$ surfaces $S\to P_1\times P_3$ branching along divisors of type $(3,0)+(0,3)$. By Proposition \ref{proposition: diagonal quotient} the surface $S$ is isomorphic to $\mu_3\bs E_1\times E_3$ with compatible $\mu_3$-operations. So we have 
\begin{equation*}
H^2(S, \QQ[\omega])_\chi\cong H^1(E_1,\QQ[\omega])_\chi\otimes H^1(E_3,\QQ[\omega])_\chi.
\end{equation*}

The branching locus of $Y\to P_1\times P_2\times P_3$ is $(C_1\times P_3)\cup (P_1\times C_2)$. Fix a point $x_i\in P_2$, $1\leq i\leq 8$, let $U$ be an open disk only containing $x_i$. We next try to describe the geometry of $p_2^{-1}(U)$. Let $\pi_1 \colon S_1\to P_1\times U$ be the triple cover of $P_1\times U$ branching along $C_1\cap (P_1\times U)$, and $\pi_2 \colon S_2\to P_1\times U$ be the triple cover of $U\times P_3$ branching along $C_2\cap (U\times P_3)$. The fibered product $S_1\times_U S_2 \to U$ has a natural action of $\mu_3\times \mu_3$ fiberwisely. The quotient by $\mu_3$ as \eqref{equation: skewed-diagonal action} gives $\mu_3\bs S_1\times_U S_2 \to U$ and it is isomorphic to $p_2^{-1}(U) \to U$. So we have
\begin{equation*}
\LL|_{U-\{x_i\}}=(R^2{p_2}_*\QQ[\omega])_\chi |_{U-\{x_i\}}\cong 
(R^1{\pi_1}_* \QQ[\omega])_\chi \otimes (R^1{\pi_2}_* \QQ[\omega])_\chi .
\end{equation*}
From Proposition \ref{proposition: isotrivial elliptic fibrations and Kodaira types}, the resolution of $S_1$ has singular fiber of type II at $x_i$ and $S_2$ is a smooth fibration. So the monodromy of $\LL$ around $x_i$ is multiplication by $-\omega^2=\exp{\pi \sqrt{-1}\over 3}$. The same arugment works for $x_i$, $9\leq i\leq 12$. So $\LL$ is a rank-one $\QQ[\omega]$ local system on $P_2-B$ corresponding to Deligne--Mostow data $\alpha=({1\over6})^{12}$.
\end{proof}

The same argument works for any type $T$ and projection $p_i$ satisfying the following condition.
\begin{cond}
\label{Condition: CY fiber rigid}
Suppose $\{1,2,3\}=\{i,i_1, i_2\}$. Any member of $T$ has component zero at $i_1$ or $i_2$. 
\end{cond}
We collect the corresponding Deligne--Mostow data in Table \ref{table: CY Deligne--Mostow}. Let $T$ be a type satisfying Condition \ref{Condition: CY fiber rigid} and $p_i\colon Y\to P_i$ corresponding to Deligne--Mostow data $\alpha$. Denote by $C$ the cyclic cover of $\PP^1$ with Deligne--Mostow data $\alpha$. 

\begin{prop}
\label{proposition: CY to Deligne--Mostow}
There is a natural isomorphism between $\QQ[\omega]$-Hodge structures of ball type
\begin{equation}
H^3(Y, \QQ[\omega])_\chi\cong H^1(C, \QQ[\omega])_{\overline{\psi}}.
	\end{equation}
\end{prop}

\begin{proof}
Since the fibers of $p_i$ are $K3$ surfaces with ADE singularities with vanishing $H^3$, we still have natural isomorphisms $H^1(P_i-B,\LL)\cong H^3(Y, \QQ[\omega])_\chi$ together with the Hermitian pairings. The other parts follow from the same argument in the proof of Proposition \ref{proposition: kondo to Deligne--Mostow}.
\end{proof}

\begin{cor}
\label{corollary: comm}
For any type $T$ in Table \ref{table: CY Deligne--Mostow} with the associated Deligne--Mostow data $\alpha$, the ball quotient corresponding to Calabi--Yau varieties of type $T$ is commensurable with ball quotient in Deligne--Mostow theory with data $\alpha$. In other words, they are the same up to a finite cover.
\end{cor}


For a fixed type $T$, it is possible that different projections $p_i$ satisfy the Condition \ref{Condition: CY fiber rigid}. Combining Corollary \ref{corollary: comm 2} and Corollary \ref{corollary: comm}, we have the following commensurability result for Deligne--Mostow lattices.

\begin{cor}
\label{corollary: comm}
The arithmetic lattices $\Gamma_\alpha\in \PU(1, n)$ with the following lists of parameters $\alpha$ are commensurable among each list.
\begin{enumerate}
  \item $n=6$, $6\alpha=((1)^8,4), ((1)^6, (2)^3)$.
  \item $n=5$, $6\alpha=(1^6,2,4), (1^7,5),(1^4,2^4)$.
  \item $n=4$, $6\alpha=(1^5,2,5), (1^4,2^2,4), (1^2,2^5)$.
  \item $n=3$, $6\alpha=(1^4, 4^2), (1^3,2^2,5), (1^2,2^3,4), (2^6)$.
  \item $n=2$, $6\alpha=(1^3,4,5), (1^2,2,4^2), (1,2^3,5), (2^4, 4)$.
\end{enumerate}
\end{cor}

\begin{rmk}
Many of the commensurability examples already appeared in \cite{sauter1988thesis, sauter1990isomorphisms, deligne1993commensurabilities} from a different point of view. For example, when the decomposition type is $(2,1)+(1, 0)+(0,2)$, the two elliptic fibrations give two weights $\mu_1=({2\over 3}, {1\over 3}, {1\over 3},{1\over 3},{1\over 3})$ and $\mu_2=({2\over 3}, {2\over 3}, {1\over 3}, {1\over 6}, {1\over 6})$. This is exactly the case in \cite[Theorem 10.6]{deligne1993commensurabilities} by setting $a=b={1\over 3}$. 
\end{rmk}

\begin{longtable}
{|p{.04\textwidth} |p{.50\textwidth} | p{.04\textwidth} | p{.18\textwidth} | p{.12\textwidth}  |}
\hline
No. &Type $T$ & $h^{2,1}$ & DM Data $6\alpha$ & Complete? \\\hline
    $1$ & $(3,3,0)+(0,0,3)$ & $9$ & $({1}^{12})$ & N \\\hline
	$2$ & $(3,2,0)+(0,1,3)$ & $9$ &  $({1}^{12})$ & N    \\\hline
	$3$ & $(3,2,0)+(0,1,2)+(0,0,1)$ & $8$ & $ ({1}^{10}, 2)$  & N  \\\hline
	$4$ & $(3,2,0)+(0,1,1)+(0,0,2)$ & $7$ &  $(1^8, 2^2)$ & N \\\hline
	$5$ & $(3,1,0)+(0,2,2)+(0,0,1)$ & $7$ & $(1^8,2^2)$ & N \\\hline
	$6$ & $(2,2,0)+(1,0,2)+(0,1,1)$ & $7$ & $(1^8,2^2)^*$ & N \\\hline
	$7$ & $(3,2,0)+(0,1,0)+(0,0,3)$ & $6$ & $(1^6,2^3)$, $(1^8,4)$ & N\\\hline
	$8$ & $(3,1,0)+(0,1,2)+(0,1,1)$ & $6$ & $(1^6, 2^3)$ & N \\\hline
	$9$ & $(3,1,0)+(0,1,0)+(0,1,3)$ & $6$ & $(1^8, 4)$ & N \\\hline
	$10$ & $(2,1,0)+(1,0,2)+(0,2,1)$ & $6$ & $(1^6,2^3)^*$ & Y \\\hline
	$11$ & $(3,1,0)+(0,2,1)+(0,0,2)$ & $5$ & $(1^4,2^4)$ & Y \\\hline
	$12$ & $(3,0,0)+(0,2,2)+(0,1,1)$ & $5$ & $(1^4,2^4)$ & N \\\hline
	$13$ & $(3,0,0)+(0,2,1)+(0,1,2)$ & $4$ & $(1^2, 2^5)$ & Y \\\hline
	$14$ & $(3,1,0)+(0,2,0)+(0,0,3)$ & $3$ & $(1^4,4^2)$, $(2^6)$ & Y \\\hline
	$15$ & $(3,0,0)+(0,3,0)+(0,0,3)$ & $0$ & $(4^3)$ & Y \\\hline
	$16$ & $(2,2,0)+(1,0,1)+(0,1,1)+(0,0,1)$ & $6$ & $(1^6,2^3)^*$ & N \\\hline
	$17$ & $(2,2,0)+(1,0,2)+(0,1,0)+(0,0,1)$ & $6$ & $(1^6,2^3)$ & N \\\hline
	$18$ & $(3,1,0)+(0,1,2)+(0,1,0)+(0,0,1)$ & $5$ & $(1^6,2,4)$ & N \\\hline
	$19$ & $(3,1,0)+(0,1,1)+(0,1,1)+(0,0,1)$ & $5$ & $(1^4,2^4)$ & Y \\\hline
	$20$ & $(2,2,0)+(1,0,1)+(0,1,0)+(0,0,2)$ & $5$ & $(1^4,2^4)$ & N \\\hline
	$21$ & $(2,1,0)+(1,1,0)+(0,1,2)+(0,0,1)$ & $5$ & $(1^4,2^4)$ & N \\\hline
	$22$ & $(2,1,0)+(1,0,2)+(0,1,0)+(0,1,1)$ & $5$ & $(1^4,2^4)^*$ & Y \\\hline
	$23$ & $(2,1,0)+(1,0,1)+(0,1,1)+(0,1,1)$ & $5$ & $(1^6, 2, 4)^*$ & Y \\\hline
	$24$ & $(2,1,0)+(1,0,1)+(0,1,2)+(0,1,0)$ & $5$ & $(1^6, 2, 4)^*$ & N \\\hline
	$25$ & $(3,0,0)+(0,2,2)+(0,1,0)+(0,0,1)$ & $4$ & $(1^4,2^4)$ & N \\\hline
	$26$ & $(2,1,0)+(1,1,0)+(0,1,1)+(0,0,2)$ & $4$ & $(1^2, 2^5)$ & Y \\\hline
	$27$ & $(2,1,0)+(1,0,2)+(0,2,0)+(0,0,1)$ & $4$ & $(1^2, 2^5)$ & Y \\\hline
	$28$ & $(3,0,0)+(0,2,1)+(0,1,1)+(0,0,1)$ & $3$ & $(2^6)$, $(1^2, 2^3, 4)$ & Y \\\hline
	$29$ & $(3,0,0)+(0,1,1)+(0,1,1)+(0,1,1)$ & $3$ & $(2^6)$ & Y \\\hline
	$30$ & $(2,1,0)+(1,0,1)+(0,2,0)+(0,0,2)$ & $3$ & $(2^6)$ & Y \\\hline
	$31$ & $(3,0,0)+(0,2,1)+(0,1,0)+(0,0,2)$ & $2$ & $(2^4, 4)$, $(1^2, 2, 4^2)$ & Y \\\hline
	$32$ & $(3,0,0)+(0,2,0)+(0,1,1)+(0,0,2)$ & $1$ & $(2^2, 4^2)$ & Y\\\hline		
	$33$ & $(2,1,0)+(0,1,2)+(1,0,0)+(0,1,0)+(0,0,1)$ & $4$ & $(1^4, 2^2, 4)$ & N \\\hline
	$34$ & $(2,1,0)+(1,0,1)+(0,1,1)+(0,1,0)+(0,0,1)$ & $4$ & $(1^4, 2^2, 4)^*$ & Y \\\hline
	$35$ & $(1,1,0)+(1,1,0)+(1,0,1)+(0,1,1)+(0,0,1)$ & $4$ & $(1^4, 2^2, 4)^*$ & Y \\\hline
	$36$ & $(2,1,0)+(1,0,0)+(0,1,1)+(0,1,0)+(0,0,2)$ & $3$ & $(1^2, 2^3, 4)$ & Y \\\hline
	$37$ & $(2,0,0)+(1,1,0)+(0,1,1)+(0,1,1)+(0,0,1)$ & $3$ & $(2^6)$ & Y \\\hline
	$38$ & $(3,0,0)+(0,1,1)+(0,1,1)+(0,1,0)+(0,0,1)$ & $2$ & $(2^4, 4)$ & Y \\\hline
	$39$ & $(2,0,0)+(1,1,0)+(0,1,1)+(0,1,0)+(0,0,2)$ & $2$ & $(2^4, 4)$ & Y \\\hline
	$40$ & $(1,1,0)+(1,0,1)+(0,1,1)+(1,0,0)+(0,1,0)+(0,0,1)$ & $3$ & $(1^2,2^3,4)^*$ & Y \\\hline
\caption{Deligne--Mostow Data in Three-Dimensional Cases}
\label{table: CY Deligne--Mostow}
\end{longtable}



The $9$ cases in Table \ref{table: CY Deligne--Mostow} with $*$ cannot be obtained directly from fibration construction. In \S\ref{section: hermitian form refinement} we calculated these cases via refinements relations. 

Similar construction and computation also works for dimension $n=4$ with $Z=(\PP^1)^4$, $d=3$, $L=\calO(1)^{\boxtimes 4}$. In Table \ref{table: dimension 4} we list the information for refinement of $T=(1, 3, 0, 0)+ (1, 0, 3, 0)+ (1, 0, 0, 3)$ which are not from half-twist of lower dimensional cases.

\begin{longtable}
{|p{.04\textwidth} |p{.50\textwidth} | p{.04\textwidth} | p{.18\textwidth} | p{.12\textwidth} |}
\hline
No. &Type $T$ & dim & DM Data $6\alpha$  & Complete? \\\hline
    $1$ & $(1,3,0,0)+(1,0,3,0)+(1,0,0,3)$ & $9$ & $({1}^{12})$ & Y \\\hline
    $2$ & $(1,3,0,0)+(1,0,3,0)+(1,0,0,2)+(0,0,0,1)$ & $8$ & $(1^{10},2)$ & Y \\\hline
    $3$ & $(1,3,0,0)+(1,0,3,0)+(1,0,0,1)+(0,0,0,2)$ & $7$ & $(1^8,2^2)$ & Y  \\\hline
    $4$ & $(1,3,0,0)+(1,0,2,0)+(1,0,0,2)+(0,0,1,0)+(0,0,0,1)$ & $7$ & $(1^8,2^2)$ & Y  \\\hline
    $5$ & $(1,3,0,0)+(1,0,2,0)+(1,0,0,1)+(0,0,1,0)+(0,0,0,2)$ & $6$ & $(1^6, 2^3)$ & Y  \\\hline
    $6$ & $(1,3,0,0)+(1,0,1,0)+(1,0,0,1)+(0,0,2,0)+(0,0,0,2)$ & $5$ & $(1^4, 2^4)$ & Y  \\\hline
    $7$ & $(1,2,0,0)+(1,0,2,0)+(1,0,0,2)+(0,1,0,0)+(0,0,1,0)+(0,0,0,1)$ & $6$ & $(1^6,2^3)$  & Y \\\hline
    $8$ & $(1,2,0,0)+(1,0,2,0)+(1,0,0,1)+(0,1,0,0)+(0,0,1,0)+(0,0,0,2)$ & $5$ & $(1^4,2^4)$ & Y  \\\hline
    $9$ & $(1,2,0,0)+(1,0,1,0)+(1,0,0,1)+(0,1,0,0)+(0,0,2,0)+(0,0,0,2)$ & $4$ & $(1^2,2^5)$  & Y \\\hline
    $10$ & $(1,1,0,0)+(1,0,1,0)+(1,0,0,1)+(0,2,0,0)+(0,0,2,0)+(0,0,0,2)$ & $3$ & $(2^6)$  & Y \\\hline

\caption{Deligne--Mostow Data in Four-Dimensional Cases}
\label{table: dimension 4}
\end{longtable}

\section{Hermitian forms and relations to Deligne--Mostow ball quotients}
\label{section: hermitian form refinement}
In \S\ref{section: 3 DM}, we relate most of the moduli spaces for the Calabi--Yau threefolds to the Deligne--Mostow ball quotients through fibrations of those threefolds. Several classes do not fall into this category because fibrations are not isotrivial. In this section, we use the results and methods of \cite{yu2024commensurability} to compute how the skew-Hermitian forms $H^3_\chi(Y)$ changes under refinements. A direct corollary is the classification of conformal classes of the Hermitian forms and complete commensurability relations to Deligne--Mostow ball quotients.

\subsection{Hermitian Forms and Refinements}
Let $Y$ be a degree-$d$ cyclic cover of $Z$ branching along divisor $D$ with $d\geq 3$. Assume that $\dim Z=n$ and $D$ is normal crossing. Recall that the Poincar\'e pairing 
\begin{equation*}
    H^n_\chi (Y, \QQ[\zeta_d])\times H^n_{\overline{\chi}}(Y, \QQ[\zeta_d])\to \QQ[\zeta_d]
\end{equation*}
induces a Hermitian form on $H^n(Y, \QQ[\zeta_d])$ when $n$ is even and a skew-Hermitian form when $n$ is odd. Let $U$ be the complement of $D$ in $Z$. We denote by $\pi^\prime=\pi|_{\pi^{-1}(U)}$. Then $\pi^\prime\colon\pi^{-1}(U)\to U$ is an unbranched degree-$d$ cyclic covering. Then the push-forward $\pi^\prime_*(\QQ[\zeta_d])$ has an action of cyclic group $\mu_d$. The character eigensheaf $\pi^\prime_*(\QQ[\zeta_d])$ is a rank-one local system $\LL$ with monodromy $\exp(-{2\pi \sqrt{-1}\over d})$ around each irreducible component of $D$. Then we have the following description of (skew-)Hermitian form $H^3_\chi(Y, \QQ[\zeta_d])$.
\begin{prop}
\label{Proposition: reduce to rk-1 local system on U}
Let $\LL^\vee$ be the dual local system of $\LL$ and it is isomorphic to the complex conjugate $\overline{\LL}$.
   \begin{enumerate}
       \item The natural map $H^n_c(U, \LL)\to H^n(U, \LL)$ is an isomorphism. 
       \item The Poincar\'e pairing
       \begin{equation*}
           H^n(U, \LL)\times H^n_c(U, \LL^\vee)\to \QQ[\zeta_d]
       \end{equation*}
       together with the isomorphism above induces a Hermitian form on $H^n(U, \LL)$ when $n$ is even, and a skew-Hermtian form when $n$ is odd.
       \item There is an isometry between (skew)-Hermitian forms
       \begin{equation*}
           H^n_\chi(Y)\cong H^n(U,\LL)\otimes \langle\gamma_0\rangle
       \end{equation*}
       where $\gamma_0$ has self-intersection $\langle\gamma_0, \gamma_0\rangle=d$.
   \end{enumerate} 
\end{prop}

\begin{proof}
See \cite[Proposition 2.6.1]{deligne1986monodromy} for dimension-one case. Let $j\colon U\to Z$ be the embedding map. Then the maps $Rj_!(\LL)[n]\to \IC(\LL)\to  Rj_*(\LL)[n]$ are isomorphisms as perverse sheaves. This is by local calculation, or more precisely, since the monodromy around each $D_j$ is nontrivial, thus the cohomology groups of nontrivial rank-one local system on the punctured discs are zero. Taking hypercohomology, we have $H^n_c(U, \LL)\to H^n(U, \LL)$. So we have a (skew)-Hermitian form on $H^n(U, \LL)$ because $$H^n_c(U, \LL^\vee)\cong H^n_c(U, \overline{\LL})\cong \overline{H^n(U, \LL)}.$$

Applying Leray-Hirsch to $Y\to Z$ and the cyclic group action, we have isomorphisms 
\[
H^n_\chi(Y, \QQ[\zeta_d])\cong H^n(U, \LL)\cong H^n_c(U, \LL).
\]
Let $F$ be a fiber of $\pi^\prime$ consisting of $d$ points and the cyclic group action. Then the Hermitian space $H^0_\chi(F, \QQ[\zeta_d])$ is generated by $\gamma_0$ with self-intersection $d$. So we have the last statement.
\end{proof}

Next we show a refinement relation for (skew)-Hermitian forms $H^n(Y, \QQ[\zeta_d])$ when the divisors $D$ degenerate. Especially, we only need the cases when $(Z,d,L)=((\PP^1)^3,3,\calO(1)^{\boxtimes 3})$. Suppose $T=(L_1,\cdots,L_m)$ is a partition of $3L$ such that $(Z,d,L,T)$ satisfies Condition \ref{condition}.
\begin{prop}
\label{proposition: hermitian form after split}
For type $T$ containing $(1,1,0)$ (respectively, $(2,1,0)$), and $T'$ obtained from $T$ by splitting $(1,1,0)$ as $(1,0,0)$ and $(0,1,0)$ (respectively, by splitting $(2,1,0)$ as $(1,1,0)$ and $(0,1,0)$). Let $Y_T$ (respectively $Y_{T^\prime}$) the Calabi--Yau orbifold of type $T$ (respectively $T^\prime$). Then there is an orthogonal sum of skew Hermitian forms:
\begin{equation*}
H^3_\chi(Y_T, \QQ[\zeta_3])=H^3_\chi(Y_{T'}, \QQ[\zeta_3])\oplus \langle \gamma_1\rangle,
\end{equation*}
where $\gamma_1\in H^3_\chi(Y_T, \QQ[\zeta_3])$ has self-intersection $\langle\gamma_1, \gamma_1\rangle=-\sqrt{-3}$.
\end{prop}

\begin{proof}
We follow the notation in \S\ref{section: refinement, incidence and complete}. We only consider the case with $L_1=\calO(1,1,0)$, $M_0=\calO(1,0,0)$, and $M_1=\calO(0,1,0)$. The other case is similar. The Calabi--Yau orbifolds $Y_T$ and $Y_{T^\prime}$ fit into a one-parameter degeneration over unit disc $\Delta=\{t\mid |t|<1\}$ constructed by Equation \eqref{equation: refinement}. The main idea is to find explicit cycles in the orthogonal complement of $H^n_\chi(Y_0)\cong H^n_\chi(Y_t)^M$ in Proposition \ref{proposition: Clemens-Schmid sequence for refinement}. 

Denote by $Z=P_1\times P_2\times P_3$. Let $D_1=\{f_1=0\}\in |\calO(1,1,0)|$ , $D_{10}=\{g_0=0\}\in |\calO(1,0,0)|$ and $D_{11}=\{g_1=0\}\in |\calO(0,1,0)|$. Denote by the branching divisor $D_t=\{(tf_1+g_0 g_1)f_2\cdots f_m=0\}$. Then $D_{10}\cap D_{11}=\{p\}\times P_3 $. Then $p$ has two open neighborhoods $V_1, V_2$ such that $V_1\subset \overline{V_1}\subset V_2$, and each $V_1$ is biholomorphic to $\Delta_1\times \Delta_2=\{(x,y)\}$. We further assume that the divisor $\{tf_1+g_0g_1=0\}$ is given by equation $xy=t$ under suitable choice of local coordinates. Then $\{p\}\times P_3\cap \{f_2\cdots f_m=0\}$ is equal to $\{p\}\times \{p_1, p_2, p_3\}$, which consists of three distinct points. For each $t\neq 0$, we denote by $U_t=Z-D_t$ and $\LL_t$ the rank-one local system on $U_t$ considered in Proposition \ref{Proposition: reduce to rk-1 local system on U}. Then the middle-dimensional cohomology of $\LL_t$ fits into the Mayer–Vietoris sequence
\begin{eqnarray*}
    &H^2((V_2-\overline{V_1})\times P_3-D_t,\LL_t)\to H^3(U_t, \LL_t)\to \\ 
    &H^3(V_2\times P_3-D_t, \LL_t)\oplus H^3(Z-(\overline{V_1}\times P_3\cup D_t), \LL_t)\to H^3((V_2-\overline{V_1})\times P_3-D_t,\LL_t).
\end{eqnarray*}
Next we calculate the monodromy $M$ on $H^3(U, \LL_t)\cong H^3_\chi(Y_t,\QQ[\zeta_3])$ when $t$ winds around zero counterclockwise by considering the monodromy on each term in the exact sequence. 

The pair $(V_2\times P_3-D_t, \LL_t)$ is isomorphic to $$((\Delta_1\times \Delta_2-\{xy=t\})\times (\PP^1- \{p_1, p_2, p_3\}), \LL_1\boxtimes\LL_2)$$ 
where $\LL_1$ and $\LL_2$ are rank-one local systems with monodromy $\exp(-{2\pi \sqrt{-1}\over 3})$ around $\{xy=t\}$ and the three punctures $p_1, p_2, p_3$ on $\PP^1$. Then we have
$$
H^3(V_2\times P_3-D_t, \LL_t)\cong H^2(\Delta_1\times \Delta_2-\{xy=t\}, \LL_1)\otimes H^1(\PP^1- \{p_1, p_2, p_3\}, \LL_2)
$$
The monodromy on $H^3(V_2\times P_3-D_t, \LL_t)$ is reduced to monodromy of $H^2(\Delta_1\times \Delta_2-\{xy=t\}, \LL_1)$. Under suitable change of coordinates, we consider $H^2(\Delta_1\times \Delta_2-\{x^2-y^2=t\}, \LL)$ where $\LL$ is a rank-one local system with monodromy $\exp(-{2\pi \sqrt{-1}\over 3})$ around $\{x^2-y^2=t\}$. Then we project to $x$-component and use Leray-Hirsch. The fibers are twice-punctured disc and there are two discriminant points with coordinates $\pm\sqrt{t}$. So 
\[
H^2(\Delta_1\times \Delta_2-\{x^2-y^2=t\}, \LL)\cong H^1(\Delta_1-\{\pm\sqrt{t}\},\LL^\prime),
\]
where $\LL^\prime$ has monodromy $\exp({\pi \sqrt{-1}\over 3})$ around $\pm\sqrt{t}$ according to \cite[Proposition 4.1]{yu2024commensurability}. So $$ \dim H^2(\Delta_1\times \Delta_2-\{x^2-y^2=t\}, \LL)=1$$ and it is generated by cycle class $\gamma_1^\prime$. 
When $t$ wind around zero counterclockwise, the two discriminant points switch their positions. So the same argument as \cite[Proposition 4.1]{yu2024commensurability} shows that the monodromy operator $M$ on $H^1(\Delta_1-\{\pm\sqrt{t}\},\LL^\prime)$ is multiplication by $\exp({2\pi \sqrt{-1}\over 3})$. The self-intersection of cycle $\gamma_1^\prime$ is 
\[
{1-\exp(-{2\pi \sqrt{-1}\over 3})\cdot\exp(-{2\pi \sqrt{-1}\over 3})\over (1-\exp(-{2\pi \sqrt{-1}\over 3}))\cdot (1-\exp(-{2\pi \sqrt{-1}\over 3}))}\cdot {1-\exp({\pi \sqrt{-1}\over 3})\cdot\exp({\pi \sqrt{-1}\over 3})\over (1-\exp({\pi \sqrt{-1}\over 3}))\cdot (1-\exp({\pi \sqrt{-1}\over 3}))}=1
\]
by \cite[Proposition 7.4]{yu2024commensurability}.

On the other hand, the constant factor $H^1(\PP^1- \{p_1, p_2, p_3\}, \LL_2)$ is generated by cycle $\gamma_1^{\prime\prime}$ with self-intersection
\[
{1-\exp(-{2\pi \sqrt{-1}\over 3})\cdot\exp(-{2\pi \sqrt{-1}\over 3})\over (1-\exp(-{2\pi \sqrt{-1}\over 3}))\cdot (1-\exp(-{2\pi \sqrt{-1}\over 3}))}=-{\sqrt{-3}\over 3}.
\]
In conclusion, the space $H^3(V_2\times P_3-D_t, \LL_t)$ is generated by $\widetilde{\gamma_1}=\gamma_1^{\prime}\otimes \gamma_1^{\prime\prime}$ with self-intersection $-{\sqrt{-3}\over 3}$ and monodromy $$M\colon \gamma_1\mapsto \exp({2\pi \sqrt{-1}\over 3})\gamma_1.$$

The pair $((V_2-\overline{V_1})\times P_3-D_t,\LL_t)$ is isomorphic to 
$$
((\Delta_1\times \Delta_2-\{xy=0\})\times (\PP^1- \{p_1, p_2, p_3\}), \LL_1^\prime\boxtimes\LL_2)
$$
where $\LL_1^\prime$ is rank-one local systems with monodromy $\exp(-{2\pi \sqrt{-1}\over 3})$ around each component of $\{xy=0\}$. Since the nontrivial rank-one local systems on punctured discs have vanishing cohomology groups, so $$H^2((V_2-\overline{V_1})\times P_3-D_t,\LL_t)=H^3((V_2-\overline{V_1})\times P_3-D_t,\LL_t)=0.$$

The Euler characteristic calculation shows that $\dim H^3_\chi(Y_T)=\dim H^3_\chi(Y_{T^\prime})+1$. Together with Proposition \ref{proposition: Clemens-Schmid sequence for refinement}, we have that $M$ acts as identity on $H^3(Z-(\overline{V_1}\times P_3\cup D_t), \LL_t)$. The orthogonal complement of $H^3_\chi(Y_{T'}, \QQ[\zeta_3])$ in $H^3_\chi(Y_T, \QQ[\zeta_3])$ is the $\exp({2\pi \sqrt{-1}\over 3})$-eigenspace of $M$, and it is generated by a cycle $\gamma_1$ corresponding to $\widetilde{\gamma_1}$. So we have the conclusion by Proposition \ref{Proposition: reduce to rk-1 local system on U}.
\end{proof}

In \S\ref{section: 3 DM}, we relate some of ball-type moduli spaces for Calabi--Yau orbifolds to Deligne--Mostow ball quotients by fibration of those Calabi--Yau orbifolds. The following proposition gives a more detailed relation on the corresponding skew-Hermitian forms.
\begin{prop}
\label{proposition: relation of hermitian forms for T and mu}
Suppose a type $T$ satisfies Condition \ref{Condition: CY fiber rigid}. Then the projection of $Y_T$ to the $i$-th factor $\PP^1$ has isotrivial K3 as fibers and induces the corresponding Deligne--Mostow tuple $\mu$. We have the following isometry between skew-Hermitian forms
\begin{equation*}
H^3_\chi(Y_T, \QQ[\zeta_3])=H_\mu\otimes \langle \gamma_2\rangle 
\end{equation*}
where $H_\mu$ the Deligne--Mostow skew-Hermitian form and $\langle\gamma_2, \gamma_2\rangle=-1$.
\end{prop}

\begin{proof}
Let $F\to \PP^1\times \PP^1$ be degree-three cyclic covering with branching divisor $D=D_1+D_2$ in linear systems $D_1\in |\calO(3,0)|$ and $D_2\in |\calO(0,3)|$. Then $\dim H^2_\chi(F)=1$ and it is generated by cycle $\gamma_2$ with self-intersection
\[
3\cdot({1-\exp(-{2\pi \sqrt{-1}\over 3})\cdot\exp(-{2\pi \sqrt{-1}\over 3})\over (1-\exp(-{2\pi \sqrt{-1}\over 3}))\cdot (1-\exp(-{2\pi \sqrt{-1}\over 3}))})^2=-1
\]
The conclusion follows from applying the Leray-Hirsch theorem to fibration $Y_T\to \PP^1$.
\end{proof}

\subsection{Conformal classes and commensurability to Deligne--Mostow lattices}
There are $9$ cases in Table \ref{table: CY Deligne--Mostow} with superscript $*$. We explain how to classify the commensurability classes of those monodromy groups and their relations to Deligne--Mostow lattices based on \cite{yu2024commensurability}. 

Recall that in \S\ref{section: hermitian form refinement}, the degree-three cyclic cover $Y$ induces a $\QQ[\zeta_3]$-valued skew-Hermitian form on $H^n(Y,\QQ[\zeta_3])$ when $n=\dim Y$ is odd. Under any $\QQ[\zeta_3]$-basis, the determinant of the corresponding Gram--Schmidt matrix gives an element in $\QQ[\zeta_3]^\times$ up to the multiplication of elements in $N_{\QQ[\zeta_3]/\QQ}(\QQ[\zeta_3]^\times)$. We denote this equivalence class by $\det (H^{n}_\chi(Y, \QQ[\zeta_3]))$.

From \cite[Theorem 6.12 (4) (5), Proposition 7.8]{yu2024commensurability}, we obtain the following criterion for commensurability relations to Deligne--Mostow lattices.

\begin{thm}
    \label{thm: conformal classes and commensurability}
Let $Y$ be a Calabi--Yau threefold of given type $T$ in Theorem \ref{theorem: main} and assume $m=h^{2,1}_\chi\geq 2$. Denote by $\Gamma\subset\PU(1,m)$ the monodromy group arising from such families. Assume $\Gamma_\mu\subset\PU(1,m)$ is a Deligne--Mostow lattices. Let $d$ be the common denominator of $\mu$. Then $\Gamma$ and $\Gamma_\mu$ are commensurable if and only if one of the following holds
\begin{enumerate}
    \item $m$ is even and $d=3,6$
    \item $m$ is odd, $d=3,6$ and $${\det H^{3}_\chi(Y, \QQ[\zeta_3])\over \det H_\mu} \in N_{\QQ[\zeta_3]/\QQ}(\QQ[\zeta_3]^\times).$$
\end{enumerate}
\end{thm}

\begin{ex}
For example, let $T=(2,1,0)+(1,0,2)+(0,2,0)+(0,0,1)$, $T'=(2,1,0)+(1,0,2)+(0,1,1)+(0,1,0)$, $T''=(2,1,0)+(1,0,2)+(0,2,1)$. By projecting to the first coordinate, we know the associated Deligne--Mostow date for $T$ is $\mu=({1\over 6}, {1\over 6}, {1\over 3}, {1\over 3}, {1\over 3}, {1\over 3}, {1\over 3})$. We have (see \cite[Proposition 7.8]{yu2024commensurability})
\[
\det H_\mu={1\over (1-\exp({\pi \sqrt{-1}\over 3}))^2 (1-\exp({2\pi \sqrt{-1}\over 3}))^5} = -{\sqrt{-3}\over 27}.
\]
By Proposition \ref{proposition: relation of hermitian forms for T and mu}, $\det H_T=-\det H_\mu = {\sqrt{-3}\over 27}$.
By Proposition \ref{proposition: hermitian form after split}, $\det H_{T'}=-\sqrt{-3} \det H_T={1\over 9}$ and this equals to 
\[
{1\over (1-\exp({\pi \sqrt{-1}\over 3}))^4 (1-\exp({2\pi \sqrt{-1}\over 3}))^4}.
\]
Then by Theorem \ref{thm: conformal classes and commensurability} we know the monodromy group $\Gamma_{T'}$ is commensurable to the Deligne--Mostow monodromy group with $\mu'= ({1\over 6}, {1\over 6}, {1\over 6}, {1\over 6}, {1\over 3}, {1\over 3}, {1\over 3}, {1\over 3})$. Similarly we know $\Gamma_{T''}$ is commensurable to $\Gamma_{\mu''}$ for $\mu'' = ({1\over 6}, {1\over 6}, {1\over 6}, {1\over 6}, {1\over 6}, {1\over 6}, {1\over 3}, {1\over 3}, {1\over 3})$.
\end{ex}

\bibliography{reference}

\Addresses
\end{document}